\documentclass[12pt, reqno]{amsart}

\allowdisplaybreaks

\usepackage{amsmath}
\usepackage{amssymb}
\usepackage{amsfonts}
\usepackage{amscd}
\usepackage{amsthm}
\usepackage{latexsym}
\usepackage{mathrsfs}
\usepackage{amsaddr}
\usepackage{color}
\usepackage{xfrac}
\usepackage{mathtools}
\usepackage{enumitem}
\usepackage{stackrel}
\usepackage{relsize}
\usepackage{hyperref}
\usepackage{xcolor}

\newcommand\al{\alpha}
\newcommand\de{\delta}
\newcommand\la{\lambda}
\newcommand{\eps}{\varepsilon}

\newcommand\C{\mathbb C}
\newcommand\kk{\Bbbk}

\renewcommand\aa{\mathfrak a}
\newcommand\bb{{\mathfrak b}}

\newcommand\gl{\mathfrak{gl}}

\newcommand\GL{\operatorname{GL}}
\newcommand\Mat{\operatorname{Mat}}
\newcommand\Rep{\operatorname{Rep}}
\newcommand\Sym{\operatorname{Sym}}
\newcommand{\VdB}{V}

\newcommand\bl{\{\!\!\{}
\newcommand\br{\}\!\!\}}

\newcommand\Atwo{A^{\otimes2}}
\newcommand\Athree{A^{\otimes3}}
\newcommand\Left{\mathbf L}
\newcommand\Right{\mathbf R}
\newcommand\OO{\mathcal O}

\newtheorem{theorem}{Theorem}[section]
\newtheorem{proposition}[theorem]{Proposition}
\newtheorem{lemma}[theorem]{Lemma}
\newtheorem{corollary}[theorem]{Corollary}

\theoremstyle{definition}
\newtheorem{definition}[theorem]{Definition}
\newtheorem{remark}[theorem]{Remark}
\newtheorem{example}[theorem]{Example}

\numberwithin{equation}{section}

\begin{document}

\title[]{Double Poisson brackets and involutive representation spaces}

\author{Grigori Olshanski${}^{1,\, 2,\, 3}$ \and Nikita Safonkin${}^{2,\,4,\, 5}$}

\thanks{
\leftline{Grigori Olshanski: olsh2007@gmail.com,  Nikita Safonkin: safonkin.nik@gmail.com}
\leftline{${}^1$ Institute for Information Transmission Problems (Kharkevich Institute), Moscow, Russia.} 
\leftline{${}^2$ Krichever Center for Advanced Studies, Skolkovo Institute of Science and Technology, Moscow, Russia.} 
\leftline{${}^3$ HSE University, Moscow, Russia.}  
\leftline{${}^4$ Institute of Mathematics, Leipzig University, Germany.}
\leftline{${}^5$ Laboratoire de Math{\'e}matiques de Reims, Universit{\'e} de Reims-Champagne-Ardenne, Reims, France.}
}

\begin{abstract}
Let $\kk$ be an algebraically closed field of characteristic $0$ and $A$ be a finitely generated associative $\kk$-algebra, in general noncommutative. One assigns to $A$ a sequence of commutative $\kk$-algebras $\OO(A,d)$, $d=1,2,3,\dots$, where $\OO(A,d)$ is the coordinate ring of the space $\Rep(A,d)$ of $d$-dimensional representations of the algebra $A$.

A  \emph{double Poisson bracket} on $A$ in the sense of Van den Bergh \cite{vdB} is a bilinear map $\bl-,-\br$ from $A\times A$ to $\Atwo$, subject to certain conditions. Van den Bergh showed that any such bracket $\bl-,-\br$ induces Poisson structures on all algebras $\OO(A,d)$. 

We propose an analog of Van den Bergh's construction, which produces Poisson structures on the coordinate rings of certain subspaces of the representation spaces $\Rep(A,d)$. We call these subspaces the \emph{involutive} representation spaces. They arise by imposing an additional symmetry condition on $\Rep(A,d)$ --- just as the classical groups from the  series B, C, D are obtained from the general linear groups (series A) as fixed point sets of involutive automorphisms.   

\end{abstract}

\date{}

\maketitle

\tableofcontents

\section{Introduction}

Throughout the paper we denote by $\kk$ a fixed algebraically closed field of characteristic $0$. By $A$ we denote a finitely generated associative algebra over $\kk$. Unless otherwise stated, we assume that $A$ contains the identity element denoted by $1$. 

\subsection{Double Poisson brackets}\label{section1.1}
Introduce a notation referring to the algebra $\Atwo=A\otimes A$. 

\begin{itemize}

\item
Elements $ x\in \Atwo$ are often written in the shorthand form $x'\otimes x''$ meaning that $x=\sum_i x'_i\otimes x''_i$ for some $x'_i, x''_i\in A$.

\item
Given $x=x'\otimes x''\in\Atwo$, we set $x^\circ:=x''\otimes x'$. So, $x\mapsto x^\circ$ is an automorphism of $\Atwo$. 

\item 
The natural action of the symmetric group $\mathfrak S_3$ on $\Athree$ is written as 
$$
s\cdot x_1\otimes x_2\otimes x_3 :=x_{s^{-1}(1)}\otimes x_{s^{-1}(2)}\otimes x_{s^{-1}(3)}, \qquad s\in\mathfrak S_3.
$$
\end{itemize}

The following definition is due to Van den Bergh (see \cite[sections 2.2 -- 2.3]{vdB}). 

\begin{definition}\label{def1.A}
{\rm(i)} A \emph{double bracket} on $A$ is a bilinear map 
$$
\bl-,-\br: A\times A \to  \Atwo 
$$
satisfying the following two conditions. 

\smallskip

$\bullet$  \emph{Skew symmetry}: for $a,b\in A$,
\begin{equation*}
\bl a,b\br=-\bl b,a\br^\circ.
\end{equation*}

$\bullet$ \emph{Leibniz rule}: for $a,b,c\in A$, 
\begin{equation*}
\bl a, bc\br=\bl a,b\br c+ b\bl a,c\br,
\end{equation*}
where, for $x=x'\otimes x''\in\Atwo$,
$$
(x'\otimes x'')c:= x'\otimes(x'' c), \qquad b(x'\otimes x''):=(bx')\otimes x''.
$$ 

{\rm(ii)} A double bracket is called \emph{double Poisson bracket} if the following relation holds.

\smallskip

$\bullet$ \emph{Double Jacobi identity}: for $a,b,c\in A$, 
\begin{equation}\label{f1}
\Bigl\{\!\!\!\Bigl\{ a,\bl b,c\br\Bigr\}\!\!\!\Bigr\}_\Left +(123)\cdot \Bigl\{\!\!\!\Bigl\{ b,\bl c,a\br\Bigr\}\!\!\!\Bigr\}_\Left +(123)^2\cdot \Bigl\{\!\!\!\Bigl\{ c,\bl a,b\br\Bigr\}\!\!\!\Bigr\}_\Left=0,
\end{equation}
where $(123)$ is the cyclic permutation $1\to 2\to 3\to 1$, which pushes everything to the right
$$
(123)\cdot (x'\otimes x''\otimes x''')=x'''\otimes x'\otimes x'',
$$
and for $x=x'\otimes x''\in\Atwo$,
$$
\bl a,x\br_\Left:= \bl a,x'\br\otimes x''\in\Athree.
$$
\end{definition}

Note that a double bracket $\bl-,-\br$ is uniquely determined by its values on any system of generators of $A$: this follows from the Leibniz rule. Here it is worth noting that, due to the symmetry condition, the following variant of the Leibniz rule also holds true:  
\begin{equation*}
\bl ab, c\br=\bl a,c\br \ast b+ a\ast \bl b,c\br,
\end{equation*}
where, for $x=x'\otimes x''\in\Atwo$,
$$
(x'\otimes x'')\ast b:= (x'b)\otimes x'', \qquad a\ast (x'\otimes x''):=x'\otimes (ax'').
$$ 

The two variants of the Leibniz rule mean that the double bracket $\bl-,-\br$ is a double derivation in its second argument for the outer bimodule structure on $A^{\otimes 2}$, as well as a double derivation in the first argument for the inner bimodule structure; see \cite[sect. 2]{Etingof} for definitions. About double brackets defined for other bimodule structures on $A^{\otimes 2}$ see \cite{FaironMcCulloch}.

\begin{example}\label{examp_2}
Let $\kk\langle\al_1,\dots,\al_L\rangle$ be the free associative algebra with $L$ generators $\al_1,\dots,\al_L$.  On it, there exists a (unique) double Poisson bracket such that   
\begin{equation*}
\bl \al_i, \al_j\br=\de_{ij}(1\otimes \al_i-\al_i\otimes 1), \quad i,j=1,\dots,L.
\end{equation*}
\end{example}

In the paper \cite{AKKN} by Alekseev, Kawazumi, Kuno, and Naef, this bracket (taken with the opposite sign) was called the \emph{Kirillov-Kostant-Souriau {\rm(}KKS{\rm)} bracket}. It is a particular case of a large family of \emph{linear} double Poisson brackets on free algebras, introduced by Pichereau and Van de Weyer \cite{pich}. Another family  (\emph{quadratic} double Poisson brackets) was investigated by Odesskii, Rubtsov and Sokolov  \cite{ORS}. These authors also considered more general, combined quadratic-linear brackets.

\subsection{The algebras $\OO(A,d)$}\label{sect1.2}

Throughout the paper $d$ denotes a positive integer. Let $\Mat(d,\kk)$ be the algebra of $d\times d$ matrices over $\kk$ and $\Rep(A,d)$ be the set of algebra homomorphisms $T: A\to \Mat(d,\kk)$ preserving the identity elements (that is, $T(1)=1$). It is an affine algebraic variety called the $d$th \emph{representation space} for the algebra $A$. 

Next, let $\OO(A,d)$ be the commutative unital $\kk$-algebra generated by the symbols $a_{ij}$ for $a\in A$ and $i,j\in\{1,\dots,d\}$, which are $\kk$-linear in $a$  and  are subject to the relations
\begin{equation*}
(ab)_{ij}=\sum_{r=1}^d a_{ir}b_{rj}, \qquad a,b\in A, \quad i,j=1,\dots,d
\end{equation*}
together with
\begin{equation*}
    1_{ij}=\delta_{ij} \qquad i,j=1,\dots,d. 
\end{equation*}
These relations mimic the evident relations satisfied by the matrix entries $T(a)_{ij}$, 
$$
T(ab)_{ij}=\sum_{r=1}^s T(a)_{ir}T(b)_{rj}, \quad a, b\in A.
$$
The algebra $\OO(A,d)$ is commonly called the \emph{coordinate ring} of the representation space $\Rep(A,d)$.  

The connection between $\Rep(A,d)$ and $\OO(A,d)$ is expressed in the fact that the elements of $\OO(A,d)$ induce functions on the space $\Rep(A,d)$, and this space can be naturally identified with $\operatorname{Specm}(\OO(A,d))$, the space of maximal ideals of $\OO(A,d)$. To strengthen this connection one should define representation spaces as affine schemes (see Le Bruyn \cite[section 2.1]{LeBruyn}). However, for our purposes this does not matter, because  we will really work only with the algebras  $\OO(A,d)$ (and their analogs).

\subsection{Van den Bergh's construction}

The following fundamental result is contained in Van den Bergh's paper \cite{vdB}, see there Propositions 1.2 and 7.5.2. 

\begin{proposition}[Van den Bergh]\label{prop1.A}
Let $A$ be an associative $\kk$-algebra and $\bl-,-\br$ be a double Poisson bracket on $A$. For each $d=1,2,\dots$, the formula
\begin{equation}\label{vdB formula}
\{a_{ij},b_{kl}\}:=\bl a,b\br'_{kj}\bl a,b\br''_{il}, \qquad a,b\in A, \quad i,j,k,l\in\{1,\dots,d\},
\end{equation}
gives rise to a Poisson bracket on the commutative algebra $\OO(A,d)$. 
\end{proposition} 

The paper \cite{vdB} was followed by many publications (a list can be found at the personal website of Maxime Fairon). Our idea is that the Van den Bergh construction can be regarded as a result ``in type A'', which raises to the problem of finding similar results in type B, C or D. As far as we know, such a problem has not been considered before.

In our understanding, Van den Bergh's construction refers to type A because it is tied to the general linear groups which play the role of symmetry groups. (A manifestation of this fact is that the group $\GL(d,\kk)$ acts, in a natural way, on the space $\Rep(A,d)$ and on the algebra $\OO(A,d)$, and the Poisson bracket $\eqref{vdB formula}$ is invariant under this action.) Our aim is to describe a construction in which the general linear groups are replaced by other classical groups, that is, the orthogonal and symplectic groups, which form the classical series B, D, and C.  

Recall that the orthogonal and symplectic groups are extracted from the general linear groups with the use of involutions. Thus, it is not surprising that involutions play an important role in our work too.  

\subsection{Involutive representation spaces and their coordinate rings}

By an \emph{involution} on an associative $\kk$-algebra we always mean a $\kk$-linear antiautomorphism whose square is the identity map. According to this, an \emph{involutive algebra} is a pair $(A,\phi)$, where $A$ is an associative algebra and $\phi$ is its involution.\footnote{Warning: this terminology diverges with the standard terminology of operator algebras, where an involution is supposed to be not a linear but antilinear map over the ground field $\kk=\C$.}

\begin{example}[Involutions on free algebras]\label{examp_1}
Here are two examples of involutions on $A=\kk\langle\al_1,\dots,\al_L\rangle$, denoted by $\phi^+$ and $\phi^-$: their action on the monomials in the generators is given by 
\begin{equation}\label{eq1.D1}
\phi^+(\al_{i_1}\dots\al_{i_k})=\al_{i_k}\dots\al_{i_1}
\end{equation}
and
\begin{equation}\label{eq1.D2}
\phi^-(\al_{i_1}\dots\al_{i_k})=(-1)^k \al_{i_k}\dots\al_{i_1}.
\end{equation}
\end{example}

\begin{example}[Involutions on matrix algebras]\label{examp_3}
Let $\langle-,-\rangle$ be a nondegenerate  bilinear form on $\kk^d\times \kk^d$, symmetric or skew-symmetric. We introduce a linear map 
$$
\tau: \Mat(d,\kk)\to \Mat(d, \kk), \qquad M\mapsto M^\tau=\tau(M),
$$
defined from the relation
$$
\langle M\xi,\eta\rangle=\langle\xi, M^\tau\eta\rangle, \qquad \xi,\eta\in\kk^d.
$$
That is, $\tau$ is the matrix transposition with respect to $\langle-,-\rangle$. Obviously, $\tau$ is an involution of the algebra $\Mat(d,\kk)$. Conversely, any involution of $\Mat(d,\kk)$ has this form (a small exercise in linear algebra). Thus, the involutions of the matrix algebras are classified according to the classification (up to natural equivalence) of the nondegenerate symmetric/skew symmetric forms. We conclude that there are three types of involutions $\tau$: type B (the form is symmetric, $d$ is odd), type C (the form is skew-symmetric, $d$ is even), and type D (the form is symmetric, $d$ is necessarily even). Note that because $\kk$ is algebraically closed, there is only one equivalence class of nondegenerate symmetric forms in each dimension $d$. 
\end{example}

\begin{definition}\label{def1} 
Let $(A,\phi)$ be an involutive algebra. Given a pair $(d,\tau)$, where $d$ is a positive integer and $\tau$ is an involution on the matrix algebra $\Mat(d,\kk)$, we  set 
\begin{equation}\label{eq1.A}
\Rep^{\phi,\tau}(A,d):=\{T\in\Rep(A,d): T\circ\phi=\tau\circ T\}.
\end{equation}
We call $\Rep^{\phi,\tau}(A,d)$ the \emph{involutive} representation space of $(A,\phi)$ corresponding to $(d,\tau)$. 
\end{definition}

The relation $T\circ\phi=\tau\circ T$ means that $T(\phi(a))=\tau(T(a))$ for every $a\in A$. Thus, elements of $\Rep^{\phi,\tau}(A,d)$ are morphisms of involutive algebras 
$$
T: (A,\phi)\to(\Mat(d,\kk),\tau)
$$ 
preserving the unity. 

Next, based on \eqref{eq1.A} and mimicking  the definition of the algebras $\OO(A,d)$, we can introduce the notion of coordinate rings in the context of involutive representation spaces. We call these new objects \emph{twisted coordinate rings} and denote them by  $\OO(A,d)^{\phi,\tau}$. These are certain quotients of the algebras $\OO(A,d)$. For instance, if $\tau$ is the conventional matrix transposition, then $\OO(A,d)^{\phi,\tau}$ is obtained from $\OO(A,d)$ by imposing the additional symmetry relations 
$$
(\phi(a))_{ij}=a_{ji}, \qquad a\in A, \quad 1\le i,j\le d.
$$

\subsection{The twisted version of Van den Bergh's construction}\label{sect1.5}

\begin{definition}\label{def_1.7}
Let $(A,\phi)$ be an involutive algebra. We say that a double Poisson bracket $\bl-,-\br$ on $A$ is \emph{$\phi$-adapted} if 
\begin{equation}\label{f2}
\phi^{\otimes2}(\bl a, b\br)=\bl \phi(a),\phi(b)\br^\circ, \qquad a,b\in A.
\end{equation}
\end{definition}

For instance, the bracket from Example \ref{examp_2} is adapted to the involution $\phi^-$ from Example \ref{examp_1} but not to the involution $\phi^+$. 

The relation \eqref{f2} means that $\phi$ is an \textit{antiautomorphism of the double bracket} (cf. the definition of morphisms of double brackets in  \cite[section 2.1.1]{FaironMorphism}).

Now we are in a position to state our main result. 

As above, we assume that $A$ is a finitely generated unital associative algebra over $\kk$, endowed with an involutive antiautomorphism $\phi$. Let $\bl-,-\br$ be an arbitrary double Poisson bracket on $A$, which is  $\phi$-adapted in the sense of Definition \ref{def_1.7}.  Next, let $\tau: \Mat(d,\kk)\to \Mat(d,\kk)$ denote an involutive antiautomorphism. Finally, let $\{E_{ij}\}$ be the natural basis of $\Mat(d,\kk)$ formed by the matrix units, and write the action of $\tau$ in this basis as
\begin{equation*}
    \tau(E_{ij})=\sum_{m,n=1}^{d} \tau_{ij}^{mn}E_{mn}.
\end{equation*}

\begin{theorem}\label{thm1.A} 
There exists a modification of the Van den Bergh's formula \eqref{vdB formula}, which gives rise to certain Poisson brackets $\{-,-\}_{\phi,\tau}$ on all twisted coordinate rings of the form $\OO(A,d)^{\phi,\tau}$. Specifically, $\{-,-\}_{\phi,\tau}$ is given on the generators by the following expression{\rm:}
\begin{equation}\label{f12}
    \{a_{ij},b_{kl}\}_{\phi,\tau}=\bl a,b\br'_{kj}\bl a,b\br''_{il}+\sum_{m,n=1}^{d}\tau_{mn}^{ij}\bl \phi(a),b\br'_{kn}\bl \phi(a),b\br''_{ml}
\end{equation}
for $a,b\in A$ and $i,j,k,l\in\{1,\dots,d\}$.
\end{theorem}

The double sum on the right-hand side can be reduced to a single term under a suitable presentation of the form $\langle-,-\rangle$ associated with $\tau$. For instance, in the case of symmetric form, one can take as $\tau$ the conventional matrix transposition, and then \eqref{f12} turns into
\begin{equation}\label{eq1.C}
\{a_{ij},b_{kl}\}_{\phi,\tau}=\bl a,b\br'_{kj}\bl a,b\br''_{il} + \bl \phi(a),b\br'_{ki}\bl \phi(a),b\br''_{jl}.
\end{equation}

Theorem \ref{thm1.A} was guessed from a simple concrete computation that we reproduce in the end of the paper, see Section \ref{sect5}.  As was shown in \cite{Ols}, the KKS bracket from Example \ref{examp_2} arises in the context of the so-called centralizer construction related to the general linear Lie algebras. We initially extended the computation from  \cite[Proposition 5.2]{Ols} to the orthogonal and symplectic Lie algebras, and this led us to the expression like the one on the right-hand side of \eqref{eq1.C}.

\subsection{Organization of the paper}

In Section \ref{sect2}, we restate the Van den Bergh construction (Proposition \ref{prop1.A} above) in a coordinate-free notation and provide a detailed proof. 

In Section \ref{sect3}, we rewrite in the same fashion the statement of Theorem \ref{thm1.A} (see Theorem \ref{thm3.A}). This enables us to give a proof in which there is no need to write down explicitly the involution $\tau: \Mat(d,\kk)\to\Mat(d,\kk)$ and distinguish between the orthogonal and symplectic cases: the proof is uniform. 

In Section \ref{sect4}, we show that there are plenty of $\phi$-adapted double Poisson brackets on the free algebras. Here we rely on results of Pichereau and Van de Weyer \cite{pich} and Odesskii-Rubtsov-Sokolov \cite{ORS}.

The final Section \ref{sect5} contains our initial computation mentioned above.

\section{The Van den Bergh  construction revisited}\label{sect2}

Before we prove the main result, Theorem \ref{thm3.A}, we would like to slightly recast Van den Bergh's proof of Proposition \ref{prop1.A}, which is Proposition 7.5.2 from \cite{vdB}. We need this in order to better explain the machinery that will be subsequently used for the twisted case, where the volume of computations increases, roughly speaking, by $4$ times due to additional summands in \eqref{f12} related to $\tau$.

\subsection{A coordinate-free notation}\label{sect2.1}

Recall that the non-twisted coordinate ring $\OO(A,d)$ was defined in section \ref{sect1.2} as an associative algebra generated by symbols $a_{ij}$, where $a\in A$ and $i,j$ are indices from $1$ to $d$. For our purposes it is convenient to describe $\OO(A,d)$ in a coordinate-free notation. 

Consider the vector space $\operatorname{Mat}^*(d,\kk)$, dual to $\operatorname{Mat}(d,\kk)$, and let $\{E^*_{ij}\}$ be the dual basis to $\{E_{ij}\}$. The space $\Mat^*(d,\kk)$ is a coalgebra with coproduct $\Delta$ and counit $\eps$ defined by
\begin{equation*}
    \Delta(E_{ij}^*):=\sum_{p=1}^d E_{ip}^*\otimes E_{pj}^*, \quad \eps(E^*_{ij}):=\de_{ij}.
\end{equation*}
We also use the shorthand notation
\begin{equation*}
    \Delta(x)=\Delta(x)'\otimes\Delta(x)'', \quad x\in\Mat^*(d,\kk).
\end{equation*}

Let us denote the vector space $A\otimes\operatorname{Mat}^*(d,\kk)$ by $L(A,d)$. 
Now we can reformulate the definition of the algebra $\OO(A,d)$ from section \ref{sect1.2}.

\begin{definition}\label{def2.A}
$\OO(A,d)$ is the commutative algebra generated by the vector space $L(A,d)$ with the following additional relations:
\begin{equation*}
    (ab)\otimes x=\left(a\otimes \Delta(x)'\right)\left(b\otimes \Delta(x)''\right), \qquad 1\otimes x=\eps(x),
\end{equation*}
where $a,b\in A$, $x\in\operatorname{Mat}^*(d,\kk)$.
\end{definition}

For $a\in A$ and $x\in \operatorname{Mat}^*(d,\kk)$, the image of $a\otimes x$ in $\OO(A,d)$ will be denoted by $(a|x)$. The correspondence between the old notation and the new notation is the following: 
\begin{equation*}
    a_{ij}=\left(a\middle|E^*_{ij}\right).
\end{equation*}

To rewrite Van den Bergh's Poisson bracket on $\OO(A,d)$ in this new notation we need to introduce an operator that is responsible for the permutation of indices $(i,j,k,l)\mapsto (k,j,i,l)$: 
\begin{equation}\label{eq2.A}
\begin{aligned}
        P\colon \operatorname{Mat}^*(d,\kk)\otimes \operatorname{Mat}^*(d,\kk)&\rightarrow \operatorname{Mat}^*(d,\kk)\otimes \operatorname{Mat}^*(d,\kk),\\
        E^*_{ij}\otimes E^*_{kl}&\mapsto E^*_{kj}\otimes E^*_{il}.
\end{aligned}
\end{equation}

\begin{remark}\label{rem3}
Here is a useful interpretation of the operator $P$. Let
    \begin{equation}\label{eq2.B}
        \sigma:=\sum\limits_{i,j} E_{ij}\otimes E_{ji}\in \operatorname{Mat}(d,
\kk)^{\otimes 2}
    \end{equation}
and let $L_{\sigma}:\operatorname{Mat}(d,\kk)^{\otimes 2}\rightarrow\operatorname{Mat}(d,\kk)^{\otimes 2}$ be the operator of left multiplication by $\sigma$ (with respect to the natural algebra structure in $\operatorname{Mat}(d,\kk)^{\otimes 2}$). It is immediately checked that $P$ is dual to $L_{\sigma}$. We use this simple remark in Proposition \ref{prop 1} below. 
\end{remark}

With this new notation in mind the Poisson bracket $\{-,-\}$ from Proposition \ref{prop1.A} can be rewritten as
\begin{equation}\label{f27}
    \left\{(a|x),(b|y)\right\}=\Bigl(\{\!\!\{a,b\}\!\!\} \mathlarger{\mathlarger{|}} P(x,y)\Bigr),
\end{equation}
where we agree that
\begin{equation}\label{f24}
    (c'\otimes c''| x\otimes y):=(c'|x)(c''|y) 
\end{equation}
and write $P(x,y)$ instead of $P(x\otimes y)$.

The next lemma ties together several useful properties of the operator $P$. Before stating them we shall recall that the symmetric group $\mathfrak{S}_3$ naturally acts on the tensor cube of $\operatorname{Mat}^*(d,\kk)$ by permuting the factors in the same way as it acts on the tensor cube of the algebra $A$, see Section \ref{section1.1}.

\begin{lemma}\label{lemma2}
{\rm(i)} Let $(12)$ stand for the permutation that swaps two tensor factors in $\operatorname{Mat}^*(d,\kk)^{\otimes 2}$, then 
        \begin{equation}\label{f16}
            (12)P=P(12).
        \end{equation}

{\rm(ii)} Consider the tensor cube of\/ $\operatorname{Mat}^*(d,\kk)$ and introduce operators $P_{12}$ and $P_{23}$ which act like $P$ on the left and right factors of the tensor cube respectively. Then 
        \begin{equation}\label{f20}
            P_{12}P_{23}(123)=(123)P_{12}P_{23}.    
        \end{equation}

{\rm(iii)} Introduce operators $\Delta_L$ and $\Delta_R$ which act like the coproduct $\Delta$ on the first and second factors of\/ $\operatorname{Mat}^*(d,\kk)^{\otimes 2}$ respectively. Then
        \begin{align}
            \Delta_L P&=P_{23}(12)\Delta_{R} \label{f25}\\
            \Delta_R P&=P_{12}\Delta_R.\label{f26}
        \end{align}    
\end{lemma}

\begin{remark}
Item (i) will be used to prove the skew-symmetry of the Poisson bracket \eqref{f27}. Item (ii)  will be used to prove the Jacobi identity for it.  Item (iii) will be helpful in proving that \eqref{f27} is compatible with the relations in $\OO(A,d)$.
\end{remark}

\begin{proof}[Proof of Lemma \ref{lemma2}]
    The relation \eqref{f16} is obvious. Let us prove \eqref{f20}. The left-hand side of \eqref{f20} evaluated on the basis element $E^*_{ij}\otimes E^*_{kl}\otimes E^*_{mn}$ of the tensor cube of $\operatorname{Mat}^*(d,\kk)$ equals
    \begin{multline*}
        P_{12}P_{23}(123)\left(E^*_{ij}\otimes E^*_{kl}\otimes E^*_{mn}\right)=P_{12}P_{23}\left(E^*_{mn}\otimes E^*_{ij}\otimes E^*_{kl}\right)=P_{12}\left(E^*_{mn}\otimes E^*_{kj}\otimes E^*_{il}\right)\\
        =E^*_{kn}\otimes E^*_{mj}\otimes E^*_{il},
    \end{multline*} 
    while the right-hand side is
    \begin{multline*}
        (123)P_{12}P_{23}\left(E^*_{ij}\otimes E^*_{kl}\otimes E^*_{mn}\right)=(123)P_{12}\left(E^*_{ij}\otimes E^*_{ml}\otimes E^*_{kn}\right)\\
        =(123)\left(E^*_{mj}\otimes E^*_{il}\otimes E^*_{kn}\right)= E^*_{kn}\otimes E^*_{mj}\otimes E^*_{il}.
    \end{multline*}

    To prove \eqref{f25} and \eqref{f26} we evaluate the left and right-hand sides on the basis element $E^*_{ij}\otimes E^*_{kl}$ of $\operatorname{Mat}^*(d,\kk)^{\otimes 2}$ and compare computation results. For \eqref{f25} we compute
    \begin{equation*}
         \Delta_L P\left(E^*_{ij}\otimes E^*_{kl}\right)=\Delta_L\left(E^*_{kj}\otimes E^*_{il}\right)=\sum\limits_{p=1}^{d}E^*_{kp}\otimes E^*_{pj}\otimes E^*_{il}
    \end{equation*}
    and 
    \begin{multline*}
        P_{23}(12)\Delta_{R}\left(E^*_{ij}\otimes E^*_{kl}\right)=\sum\limits_{p=1}^{d} P_{23}(12)\left(E^*_{ij}\otimes E^*_{kp}\otimes E^*_{pl}\right)=\sum\limits_{p=1}^{d} P_{23}\left(E^*_{kp}\otimes E^*_{ij}\otimes E^*_{pl}\right)\\
        =\sum\limits_{p=1}^{d} E^*_{kp}\otimes E^*_{pj}\otimes E^*_{il}.
    \end{multline*}

    To prove \eqref{f26} we do the following computations
    \begin{equation*}
        \Delta_R P\left(E^*_{ij}\otimes E^*_{kl}\right)=\Delta_R\left(E^*_{kj}\otimes E^*_{il}\right)=\sum\limits_{p=1}^{d}E^*_{kj}\otimes E^*_{ip}\otimes E^*_{pl}
    \end{equation*}
    and
    \begin{equation*}
        P_{12}\Delta_R\left(E^*_{ij}\otimes E^*_{kl}\right)=\sum\limits_{p=1}^{d}P_{12}\left(E^*_{ij}\otimes E^*_{kp}\otimes E^*_{pl}\right)=\sum\limits_{p=1}^{d} E^*_{kj}\otimes E^*_{ip}\otimes E^*_{pl}.
    \end{equation*}
\end{proof}

The proof of Proposition \ref{prop1.A} is divided into two parts, which occupy the next two subsections. By the very definition of the coordinate ring $\OO(A,d)$ (see subsection \ref{sect2.1}), it is a quotient of the symmetric algebra $S(L(A,d))$. We first construct a Poisson bracket on $S(L(A,d))$ and then check that it can be descended to $\OO(A,d)$.

\subsection{Poisson bracket on $S(L(A,d))$}

For any double bracket $\{\!\!\{-,-\}\!\!\}$ on $A$ we set (see \cite[section 2.3]{vdB})
\begin{equation}\label{f49}
    \{\!\!\{a,b,c\}\!\!\}=\Bigl\{\!\!\!\Bigl\{a,\{\!\!\{b,c\}\!\!\}\Bigr\}\!\!\!\Bigr\}_\Left+(123)\cdot\Bigl\{\!\!\!\Bigl\{b,\{\!\!\{c,a\}\!\!\}\Bigr\}\!\!\!\Bigr\}_\Left+(123)^2\cdot\Bigl\{\!\!\!\Bigl\{c,\{\!\!\{a,b\}\!\!\}\Bigr\}\!\!\!\Bigr\}_\Left.
\end{equation}
In this notation, the double Jacobi identity \eqref{f1} is written as $\bl a,b,c\br=0$.

The next proposition deals with a Poisson bracket on the symmetric algebra $S\left(L(A,d)\right)$ but not on its quotient $\OO(A,d)$. In this proposition we use the notation
\begin{equation*}
    c\overset{2}{\otimes} x:=(c'\otimes x')(c''\otimes x'')\in S^2\left(L(A,d)\right),
\end{equation*}
where $c=c'\otimes c''\in A^{\otimes 2}$ and $x=x'\otimes x''\in\operatorname{Mat}^*(d,\kk)^{\otimes 2}$. We adopt a similar convention for the third symmetric power $S^3\left(L(A,d)\right)$ together with $c\in A^{\otimes 3}$ and $x\in\operatorname{Mat}^*(d,\kk)^{\otimes 3}$ as well, i.e.
\begin{equation*}
    c\overset{3}{\otimes} x:=(c'\otimes x')(c''\otimes x'')(c'''\otimes x''')\in S^3\left(L(A,d)\right).
\end{equation*}

Recall that the symmetric group $\mathfrak{S}_3$ naturally acts both on $A^{\otimes 3}$ and $\operatorname{Mat}^*(d,\kk)^{\otimes 3}$.

Due to commutativity of the symmetric algebra we have
\begin{equation}\label{f19}
(s\cdot c)\overset{3}{\otimes} x=c\overset{3}{\otimes}(s^{-1}\cdot x),    
\end{equation}
for $s\in\mathfrak S_3$, $c\in A^{\otimes 3}$, and $x\in\operatorname{Mat}^*(d,\kk)^{\otimes 3}$.

Recall that $L(A,d)$ stands for the vector space $A\otimes\operatorname{Mat}^*(d,\kk)$.

\begin{proposition}\label{th2}
    Let $A$ be an associative $\kk$-algebra and $\bl-,-\br$ be a double Poisson bracket on $A$. For each $d=1,2,\dots$, the formula
\begin{equation}\label{vdB formula.2}
\left\{a\otimes x,b\otimes y\right\}=\{\!\!\{a,b\}\!\!\}\overset{2}{\otimes} P(x,y), \qquad a,b\in A, \quad x,y\in \operatorname{Mat}^*(d,\kk),
\end{equation}
gives rise to a Poisson bracket on the commutative algebra $S\left(L(A,d)\right)$. 
\end{proposition}

\begin{remark}\label{rem2.A}
The fact that the bracket \eqref{vdB formula.2} satisfies the Jacobi identity will follow from a more general claim. Namely, let  $\{\!\!\{-,-\}\!\!\}$ be a double bracket on $A$, which is not necessarily a double Poisson bracket, and let $f_1=a\otimes x$, $f_2=b\otimes y$, and $f_3=c\otimes z$ be elements of $L(A,d)$. Then the following identity holds in $S\left(L(A,d)\right)$:
\begin{multline}\label{f13}
    \Bigl\{f_1,\{f_2,f_3\}\Bigr\}+\Bigl\{f_2,\{f_3,f_1\}\Bigr\}+\Bigl\{f_3,\{f_1,f_2\}\Bigr\}\\
    =\{\!\!\{a,b,c\}\!\!\}\overset{3}{\otimes} P_{12}P_{23}\left(x,y,z\right)-\{\!\!\{a,c,b\}\!\!\}\overset{3}{\otimes}P_{12}P_{23}\left(x,z,y\right),
\end{multline}
where $P_{12}P_{23}(x_1,x_2,x_3):=P\left(x_1,P\left(x_2,x_3\right)'\right)\otimes P\left(x_2,x_3\right)''\in\operatorname{Mat}^*(d,\kk)^{\otimes 3}$.
\end{remark}

\begin{proof}[Proof of Proposition \ref{th2}]
    First of all we shall note that the right hand side of \eqref{vdB formula.2} is skew symmetric due to the skew symmetry property of $\bl-,-\br$  (Definition \ref{def1.A}), commutativity of the symmetric algebra, and \eqref{f16}. 
    
 Now we prove \eqref{f13}. By the very definition we have
    \begin{equation*}
        \{f_2,f_3\}=\{\!\!\{b,c\}\!\!\}\overset{2}{\otimes}P(y,z)=\Bigl(\{\!\!\{b,c\}\!\!\}'\otimes P(y,z)'\Bigr)\Bigl(\{\!\!\{b,c\}\!\!\}''\otimes P(y,z)''\Bigr).
    \end{equation*}
    Then by the Leibniz rule we have
    \begin{align}\label{f17}
     \notag
            \Bigl\{f_1,\{f_2,f_3\}\Bigr\}&=\Bigl\{f_1,\{\!\!\{b,c\}\!\!\}'\otimes P(y,z)'\Bigr\}\Bigl(\{\!\!\{b,c\}\!\!\}''\otimes P(y,z)''\Bigr)\\ \notag
            &\phantom{aaaaaaa}+\Bigl(\{\!\!\{b,c\}\!\!\}'\otimes P(y,z)'\Bigr)\Bigl\{f_1,\{\!\!\{b,c\}\!\!\}''\otimes P(y,z)''\Bigr\}\\
            &=\left(\Bigl\{\!\!\!\Bigl\{a,\{\!\!\{b,c\}\!\!\}'\Bigr\}\!\!\!\Bigr\}\overset{2}{\otimes} P(x,P(y,z)')\right)\Bigl(\{\!\!\{b,c\}\!\!\}''\otimes P(y,z)''\Bigr)\\ \notag
            &\phantom{aaaaaaa}+\Bigl(\{\!\!\{b,c\}\!\!\}'\otimes P(y,z)'\Bigr)\left(\Bigl\{\!\!\!\Bigl\{a,\{\!\!\{b,c\}\!\!\}''\Bigr\}\!\!\!\Bigr\}\overset{2}{\otimes} P(x,P(y,z)'')\right).
    \end{align}
    Then we note that the first summand equals 
    \begin{equation*}
        \Bigl\{\!\!\!\Bigl\{a,\{\!\!\{b,c\}\!\!\}\Bigr\}\!\!\!\Bigr\}_{\Left}\overset{3}{\otimes} P_{12}P_{23}(x,y,z).
    \end{equation*}
    By the skew-symmetry of the double bracket $\{\!\!\{-,-\}\!\!\}$ we have
    \begin{equation*}
        \{\!\!\{b,c\}\!\!\}'\otimes \{\!\!\{b,c\}\!\!\}''=-\{\!\!\{c,b\}\!\!\}''\otimes \{\!\!\{c,b\}\!\!\}',
    \end{equation*}
    then we can rewrite the second summand in \eqref{f17} as follows
    \begin{equation}\label{f18}
        -\Bigl\{\!\!\!\Bigl\{a,\{\!\!\{c,b\}\!\!\}\Bigr\}\!\!\!\Bigr\}_{\Left}\overset{3}{\otimes} P_{12}(23)P_{23}(x,y,z),
    \end{equation}
    where $(23)$ is the permutation that swaps the second and the third factors in the tensor cube of $\operatorname{Mat}^*(d,\kk)$. Next, we use \eqref{f16} to rewrite \eqref{f18} in the desired form. 
    
    Now we can tie all the things together
    \begin{multline}\label{f21}
        \Bigl\{f_1,\{f_2,f_3\}\Bigr\}=\Bigl\{\!\!\!\Bigl\{a,\{\!\!\{b,c\}\!\!\}\Bigr\}\!\!\!\Bigr\}_{\Left}\overset{3}{\otimes} P_{12}P_{23}(x,y,z)\\
        -\Bigl\{\!\!\!\Bigl\{a,\{\!\!\{c,b\}\!\!\}\Bigr\}\!\!\!\Bigr\}_{\Left}\overset{3}{\otimes} P_{12}P_{23}(23)(x,y,z).
    \end{multline}
    
    Now we can apply \eqref{f21} to various inputs in order to obtain similar expressions for $\Bigl\{f_2,\{f_3,f_1\}\Bigr\}$ and $\Bigl\{f_3,\{f_1,f_2\}\Bigr\}$. Notice the permutations $(123)$ and $(123)^2$ inside the right tensor factor in formulas below
    \begin{multline*}
        \Bigl\{f_2,\{f_3,f_1\}\Bigr\}=\Bigl\{\!\!\!\Bigl\{b,\{\!\!\{c,a\}\!\!\}\Bigr\}\!\!\!\Bigr\}_{\Left}\overset{3}{\otimes} P_{12}P_{23}(123)^2(x,y,z)\\
        -\Bigl\{\!\!\!\Bigl\{b,\{\!\!\{a,c\}\!\!\}\Bigr\}\!\!\!\Bigr\}_{\Left}\overset{3}{\otimes} P_{12}P_{23}(23)(123)^2(x,y,z)
    \end{multline*}
    and
    \begin{multline*}
        \Bigl\{f_3,\{f_1,f_2\}\Bigr\}=\Bigl\{\!\!\!\Bigl\{c,\{\!\!\{a,b\}\!\!\}\Bigr\}\!\!\!\Bigr\}_{\Left}\overset{3}{\otimes} P_{12}P_{23}(123)(x,y,z)\\
        -\Bigl\{\!\!\!\Bigl\{c,\{\!\!\{b,a\}\!\!\}\Bigr\}\!\!\!\Bigr\}_{\Left}\overset{3}{\otimes} P_{12}P_{23}(23)(123)(x,y,z).
    \end{multline*}

     We would like to move the permutations $(123)$ and $(123)^2$ from the second tensor factors to the first ones. For that we recall
     \begin{equation*}
       (23)(123)=(123)^2(23) 
     \end{equation*}
     and use \eqref{f20} and \eqref{f19} to obtain 
     \begin{multline}\label{f22}
        \Bigl\{f_2,\{f_3,f_1\}\Bigr\}=(123)\cdot\Bigl\{\!\!\!\Bigl\{b,\{\!\!\{c,a\}\!\!\}\Bigr\}\!\!\!\Bigr\}_{\Left}\overset{3}{\otimes} P_{12}P_{23}(x,y,z)\\
        -(123)^2\cdot\Bigl\{\!\!\!\Bigl\{b,\{\!\!\{a,c\}\!\!\}\Bigr\}\!\!\!\Bigr\}_{\Left}\overset{3}{\otimes} P_{12}P_{23}(x,z,y)
    \end{multline}
    and
    \begin{multline}\label{f23}
        \Bigl\{f_3,\{f_1,f_2\}\Bigr\}=(123)^2\cdot\Bigl\{\!\!\!\Bigl\{c,\{\!\!\{a,b\}\!\!\}\Bigr\}\!\!\!\Bigr\}_{\Left}\overset{3}{\otimes} P_{12}P_{23}(x,y,z)\\
        -(123)\cdot\Bigl\{\!\!\!\Bigl\{c,\{\!\!\{b,a\}\!\!\}\Bigr\}\!\!\!\Bigr\}_{\Left}\overset{3}{\otimes} P_{12}P_{23}(x,z,y).
    \end{multline}

    Finally, summing \eqref{f21}, \eqref{f22}, and \eqref{f23} we get \eqref{f13}.
\end{proof}

\subsection{Poisson bracket on $\OO(A,d)$}

Recall that $(a|x)$ for $a\in A$ and $x\in\operatorname{Mat}^*(d,\kk)$ is the image of $a\otimes x\in L(A,d)$ in $\OO(A,d)$. Recall also notation \eqref{f24}. The next result is a reformulation of Proposition \ref{prop1.A}. 

\begin{proposition}\label{th1.B}
    Let $A$ be an associative $\kk$-algebra and $\bl-,-\br$ be a double Poisson bracket on $A$. For each $d=1,2,\dots$, the formula \eqref{vdB formula.2}, rewritten as
\begin{equation}\label{eq2.E}
\left\{(a|x),(b|y)\right\}=\Bigl(\{\!\!\{a,b\}\!\!\} \mathlarger{\mathlarger{|}} P(x,y)\Bigr), \qquad a,b\in A, \quad x,y\in \operatorname{Mat}^*(d,\kk),
\end{equation}
gives rise to a Poisson bracket on the commutative algebra $\OO(A,d)$. 
\end{proposition}

\begin{proof}
Recall (Definition \ref{def2.A}) that $\OO(A,d)$ is the quotient of $S(L(A,d))$ by the ideal generated by the elements  
$$
(bc)\otimes y-\left(b\otimes \Delta(y)'\right)\left(c\otimes \Delta(y)''\right) \quad \text{and} \quad 1\otimes y- \eps(y),
$$
where $b, c\in A$, $y\in\Mat^*(d,\kk)$. 

Denote this ideal by $I$ and set $\operatorname{Sym}(A,d):=S(L(A,d))$. By virtue of Proposition \ref{th2}, it suffices to prove that $I$ is a Poisson ideal in $\Sym(A,d)$, that is $\{X,Y\}\in I$ for any $X\in \Sym(A,d)$ and $Y\in I$.  Next, it suffices to check this on the generators.  For $Y=1\otimes y- \eps(y)$ this holds for trivial reasons, because 
$\{a\otimes x,1\otimes y \}=\{a\otimes x, \eps(y)\}=0$ for any $a\in A$ and $x, y\in\Mat^*(d,\kk)$. 

Now take arbitrary $a,b,c\in A$ and $x,y\in\operatorname{Mat}^*(d,\kk)$,  and check that
    \begin{equation}\label{eq2.C}
         \Bigl\{a\otimes x, \; (bc)\otimes y-\left(b\otimes \Delta(y)'\right)\left(c\otimes \Delta(y)''\right)\Bigr\}\in I.
    \end{equation}
 By the very definition we have
\begin{equation*}
    \{a\otimes x, (bc)\otimes y\}=\{\!\!\{a,bc\}\!\!\}\overset{2}{\otimes}P(x,y).
\end{equation*}
Then by the Leibniz rule for the double bracket we have
\begin{equation*}
    \{\!\!\{a,bc\}\!\!\}=\{\!\!\{a,b\}\!\!\}c+b\{\!\!\{a,c\}\!\!\}. 
\end{equation*}
So,
\begin{multline}\label{f14}
     \{a\otimes x,(bc)\otimes y\}=\{\!\!\{a,b\}\!\!\}c\overset{2}{\otimes} P(x,y)+b\{\!\!\{a,c\}\!\!\}\overset{2}{\otimes} P(x,y)\\
     =\Bigl(\{\!\!\{a,b\}\!\!\}'\otimes P(x,y)'\Bigr)\Bigl(\{\!\!\{a,b\}\!\!\}''c\otimes P(x,y)''\Bigr)\\
     +\Bigl(b\{\!\!\{a,c\}\!\!\}'\otimes P(x,y)'\Bigr)\Bigl(\{\!\!\{a,c\}\!\!\}''\otimes P(x,y)''\Bigr).
\end{multline}

Take a closer look at the first summand's second factor in \eqref{f14}. It contains an element $\{\!\!\{a,b\}\!\!\}''c$, which we would like to divide into two parts by replacing the expression with the product of two elements of $\operatorname{Sym}(A,d)$ using a relation from the ideal $I$. 

Below the symbol ``$\underset{I}{=}$'' means ``equality modulo $I$''.   

By the very definition of the ideal $I$ we can write
    \begin{equation*}
        \{\!\!\{a,b\}\!\!\}''c\otimes P(x,y)''\underset{I}{=}\Bigl(\{\!\!\{a,b\}\!\!\}''\otimes \Delta(P(x,y)'')'\Bigr)\Bigl(c\otimes \Delta(P(x,y)'')''\Bigr)
    \end{equation*}
    and
    \begin{equation*}
        b\{\!\!\{a,c\}\!\!\}'\otimes P(x,y)'\underset{I}{=}\Bigl(b\otimes \Delta(P(x,y)')'\Bigr)\Bigl(\{\!\!\{a,c\}\!\!\}'\otimes \Delta(P(x,y)')''\Bigr).
    \end{equation*}
Here on the left-hand sides we have just referred to the second Sweedler's component without the first one. This is an illegal operation, but the expressions from the right-hand sides will only be used to rewrite \eqref{f14} and there will be no problem with that. 

    So,
    \begin{multline}\label{f15}
        \{a\otimes x,(bc)\otimes y\}\underset{I}{=}\Bigl(\{\!\!\{a,b\}\!\!\}'\otimes P(x,y)'\Bigr)\Bigl(\{\!\!\{a,b\}\!\!\}''\otimes \Delta(P(x,y)'')'\Bigr)\Bigl(c\otimes \Delta(P(x,y)'')''\Bigr)\\
     +\Bigl(b\otimes \Delta(P(x,y)')'\Bigr)\Bigl(\{\!\!\{a,c\}\!\!\}'\otimes \Delta(P(x,y)')''\Bigr)\Bigl(\{\!\!\{a,c\}\!\!\}''\otimes P(x,y)''\Bigr).
    \end{multline}

 Next,  we rewrite the right-hand side of \eqref{f15} in a more compact form, which gives
 \begin{equation}\label{eq2.D}
        \{a\otimes x,(bc)\otimes y\}\underset{I}{=}\Bigl(\{\!\!\{a,b\}\!\!\}\otimes c\Bigr)\overset{3}{\otimes}\Delta_R(P(x,y))+\Bigl(b\otimes \{\!\!\{a,c\}\!\!\}\Bigr)\overset{3}{\otimes}\Delta_{L}(P(x,y)),
    \end{equation}
    where $\Delta_L$ (respectively, $\Delta_R$) is the coproduct $\Delta$ applied to the first (respectively, second) factor in $\operatorname{Mat}^*(d,\kk)^{\otimes 2}$.

    Let us handle the first summand on the right-hand  of \eqref{eq2.D}. By \eqref{f26} we can write 
    \begin{equation*}
        \Delta_R P(x,y)=P_{12}\Delta_R(x,y)=P_{12}\left(x\otimes \Delta(y)\right)=P(x,\Delta(y)')\otimes \Delta(y)''.
    \end{equation*}
It follows that that the first summand equals
   \begin{equation*}
       \Bigl(\{\!\!\{a,b\}\!\!\}\overset{2}{\otimes} P(x,\Delta(y)')\Bigr)(c\otimes \Delta(y)''),
   \end{equation*}
   which is 
   \begin{equation*}
       \{a\otimes x,\; b\otimes\Delta(y)'\}(c\otimes \Delta(y)'').
   \end{equation*}

Now we turn to the second summand in \eqref{eq2.D}. By \eqref{f25} we have
   \begin{equation*}
       \Delta_L P(x,y)=P_{23}(12)(x\otimes \Delta(y))=P_{23}(\Delta(y)'\otimes x\otimes \Delta(y)'')=\Delta(y)'\otimes P(x,\Delta(y)'').
   \end{equation*}
It follows that the second summand equals
   \begin{equation*}
       (b\otimes \Delta(y)')\Bigl(\{\!\!\{a,c\}\!\!\}\overset{2}{\otimes} P(x,\Delta(y)'')\Bigr),
   \end{equation*}
   which is 
   \begin{equation*}
       (b\otimes \Delta(y)')\{a\otimes x,\; c\otimes \Delta(y)''\}.
   \end{equation*}
 This completes the proof of \eqref{eq2.C}. 
\end{proof}

\section{Main theorem: the twisted version of Van den Bergh's construction}\label{sect3}

\subsection{Preliminaries}
Our goal is to prove Theorem \ref{thm3.A} which is a reformulation of Theorem \ref{thm1.A} from the introduction. First of all we shall give a precise definition of the twisted coordinate ring mentioned below Definition \ref{def1}.

Recall that $\phi$ and $\tau$ are involutive antiautomorphisms of $A$ and $\operatorname{Mat}(d,\kk)$ respectively. We denote by $\tau^*\colon\operatorname{Mat}^*(d,\kk)\rightarrow \operatorname{Mat}^*(d,\kk)$ the dual map to $\tau$.  

Recall that the element of $\OO(A,d)$ corresponding to $a\otimes x\in L(A,d)$ is denoted by $(a|x)\in \OO(A,d)$. 

\begin{definition}\label{def2}
The \emph{twisted coordinate ring} $\OO(A,d)^{\phi,\tau}$ is the quotient of $\OO(A,d)$ by the ideal generated by the elements
 \begin{equation}\label{eq3.F}
 (\phi(a)|x)-(a|\tau^*(x)), \qquad a\in A, \quad x\in\operatorname{Mat}^*(d,\kk).
 \end{equation}
\end{definition}

With an abuse of notation we will denote the element of $\OO(A,d)^{\phi,\tau}$ corresponding to $(a|x)\in\OO(A,d)$ by the same symbol $(a|x)$. We also extend the notation introduced above Proposition \ref{th2} to the case of the twisted coordinate ring.

Then the bracket from Theorem \ref{thm1.A} can be rewritten as
\begin{equation}\label{f52}
    \{(a|x),(b|y)\}_{\phi,\tau}=\Bigl(\{\!\!\{a,b\}\!\!\}\mathlarger{|}P(x,y)\Bigr)+\Bigl(\{\!\!\{\phi(a),b\}\!\!\}\mathlarger{|}P(\tau^*(x),y)\Bigr).
\end{equation}
The first summand here corresponds to the first summand in \eqref{f12} and the second summand corresponds to the double sum in \eqref{f12}.
\begin{remark}\label{rem2}
    It is clear that the right-hand side of \eqref{f52} is consistent with the relation $(\phi(a)|x)=(a|\tau^*(x))$ because if we substitute $\phi(a)$ for $a$ and $\tau^*(x)$ for $x$ then the two summands are swapped, and their sum remains unchanged.
\end{remark}

We need a twisted analog of Proposition \ref{th2}. For this purpose we define an analog of the vector space $L(A,d)=A\otimes\operatorname{Mat}^*(d,\kk)$.

\begin{definition}
We denote by $L(A,d)^{\phi,\tau}$ the quotient of the vector space $L(A,d)$ by the linear span of the elements
$$
\phi(a)\otimes x-a\otimes \tau^*(x), \qquad a\in A, \quad x\in\operatorname{Mat}^*(d,\kk).
$$    
\end{definition}

Elements of this quotient space will be denoted by the same symbols $a\otimes x$. There will be no confusion because we will work only with $L(A,d)^{\phi,\tau}$ but not with $L(A,d)$. So, in $L(A,d)^{\phi,\tau}$ one has
\begin{equation}\label{f30}
    \phi(a)\otimes x=a\otimes \tau^*(x).
\end{equation}

We will call \eqref{f30} \textit{the basic relation} in $L(A,d)^{\phi,\tau}$.

Below we will use a notation similar to that given above Proposition \ref{th2}. For instance, we set
\begin{equation*}
    c\overset{2}{\otimes} x:=(c'\otimes x')(c''\otimes x'')\in S^2\left(L(A,d)^{\phi,\tau}\right),
\end{equation*}
for any $c=c'\otimes c''\in A^{\otimes 2}$ and $x=x'\otimes x''\in\operatorname{Mat}^*(d,\kk)^{\otimes 2}$.

\begin{proposition}\label{th1}
    Let $(A,\phi)$ be an involutive $\kk$-algebra and $\{\!\!\{-,-\}\!\!\}$ be a double Poisson bracket on $A$ which is $\phi$-adapted. Let $\tau$ be an involution on the matrix algebra $\operatorname{Mat}(d,\kk)$. For each $d=1,2,\dots$, the formula
\begin{equation}\label{f29}
\left\{a\otimes x,b\otimes y\right\}_{\phi,\tau}:=\{\!\!\{a,b\}\!\!\}\overset{2}{\otimes} P(x,y)+\{\!\!\{\phi(a),b\}\!\!\}\overset{2}{\otimes} P(\tau^*(x),y),
\end{equation}
where $a,b\in A$ and $x,y\in\operatorname{Mat}^*(d,\kk)$, gives rise to a Poisson bracket on the commutative algebra $S\left(L(A,d)^{\phi,\tau}\right)$.
\end{proposition}

\begin{remark}[cf. Remark \ref{rem2.A}]
Let $\VdB(a,b,c;x,y,z)$ stand for the right-hand side of \eqref{f13}, i.e.
\begin{equation*}
    \VdB(a,b,c;x,y,z)=\{\!\!\{a,b,c\}\!\!\}\overset{3}{\otimes} P_{12}P_{23}\left(x,y,z\right)-\{\!\!\{a,c,b\}\!\!\}\overset{3}{\otimes}P_{12}P_{23}(23)\left(x,y,z\right),
\end{equation*}
where $\{\!\!\{-,-,-\}\!\!\}$ is the ternary operation defined in \eqref{f49}. 
The fact that the bracket \eqref{f29} satisfies the Jacobi identity will follow from a more general claim. Namely, let  $\{\!\!\{-,-\}\!\!\}$ be a double bracket on $A$, which is not necessarily a double Poisson bracket, and let $f_1=a\otimes x$, $f_2=b\otimes y$, and $f_3=c\otimes z$ be elements of $L(A,d)^{\phi,\tau}$. Then the following identity holds in $S\left(L(A,d)^{\phi,\tau}\right)$:
\begin{multline}\label{f28}
    \Bigl\{f_1,\{f_2,f_3\}_{\phi,\tau}\Bigr\}_{\phi,\tau}+\Bigl\{f_2,\{f_3,f_1\}_{\phi,\tau}\Bigr\}_{\phi,\tau}+\Bigl\{f_3,\{f_1,f_2\}_{\phi,\tau}\Bigr\}_{\phi,\tau}\\
    =\VdB(a,b,c;x,y,z)+\VdB(\phi(a),b,c;\tau^*(x),y,z)+\VdB(a,\phi(b),c;x,\tau^*(y),z)\\
    +\VdB(a,b,\phi(c);x,y,\tau^*(z)).
\end{multline}
\end{remark}

To prove Proposition \ref{th1} we will need several auxiliary claims.

\begin{lemma}\label{prop 1}
    Let, as before, $\tau$ be an involutive antiautomorphism of the matrix algebra $\operatorname{Mat}(d,\kk)$. For any $x,y\in\operatorname{Mat}^*(d,\kk)$ we have
    \begin{equation}\label{f33}
        (\tau^*)^{\otimes 2}P(x,y)=(12)P(\tau^*(x),\tau^*(y)).
    \end{equation}
\end{lemma}

\begin{proof}
    Let us rewrite \eqref{f33} in the dual terms using Remark \ref{rem3}
    \begin{equation}
        L_{\sigma}\left(\tau\otimes\tau\right)=\left(\tau\otimes\tau\right) L_{\sigma}(12).
    \end{equation}

    Let $a,b\in\operatorname{Mat}(d,\kk)$. The left-hand side evaluated at this pair equals
    \begin{equation}
        \operatorname{LHS}(a,b):=\sigma \left(\tau(a)\otimes\tau(b)\right)
    \end{equation}
    while the right-hand side gives
    \begin{equation}
        \operatorname{RHS}(a,b):=\left(\tau\otimes\tau\right)(\sigma (b\otimes a))=(\tau(b)\otimes\tau(a))\cdot\tau^{\otimes 2}(\sigma).
    \end{equation}

    Now we note that $\tau^{\otimes 2}(\sigma)=\sigma$, which is verified as follows:
    \begin{enumerate}[label=\theenumi)]
        \item The antiautomorphism $\tau$ can be written in the form $A\mapsto gA^tg^{-1}$ for an invertible matrix $g\in\operatorname{GL}(d,\kk)$, where $t$ stands for the conventional matrix transposition (indeed, this follows from the fact that any automorphism of the matrix algebra is inner).

        \item $t^{\otimes 2}(\sigma)=\sigma$.

        \item $(a_1\otimes a_2)\sigma=\sigma(a_2\otimes a_1)$ for any matrices $a_1,a_2\in\operatorname{Mat}(d,\kk)$ because $\sigma\in\operatorname{Mat}(d,\kk)^{\otimes 2}$ is the operator that swaps two tensor factors in $(\kk^d)^{\otimes 2}$. In particular, $(g\otimes g)\sigma (g^{-1}\otimes g^{-1})=\sigma$ for any invertible matrix $g\in\operatorname{GL}(d,\kk)$.
    \end{enumerate}

    So, we have 
    \begin{equation}
        \operatorname{RHS}(a,b)=(\tau(b)\otimes\tau(a))\cdot \sigma,
    \end{equation}
    which equals $\operatorname{LHS}(a,b)$ due to the third property from the list above.
\end{proof}

The following corollary is a technical tool which we will use in the proof of Proposition \ref{th1}.

\begin{corollary}\label{corollary1}
    We have
    \begin{equation*}
        (12)(\tau^*)^{\otimes 3}P_{12}P_{23}=P_{12}P_{23}(23)(\tau^*)^{\otimes 3}.
    \end{equation*}
\end{corollary}
\begin{proof}
    Let us evaluate the left-hand side on $x\otimes y\otimes z\in\operatorname{Mat}^*(d,\kk)^{\otimes 3}$.
    \begin{multline*}
        (12)(\tau^*)^{\otimes 3}P_{12}P_{23}(x,y,z)=(12)(\tau^*)^{\otimes 3}\left(P(x,P(y,z)')\otimes P(y,z)''\right)\\
        =(12)(\tau^*)^{\otimes 2} P(x,P(y,z)')\otimes \tau^* \left(P(y,z)''\right).
    \end{multline*}
    Then we apply Lemma \ref{prop 1} to the left tensor factor and obtain
    \begin{equation*}
        (12)(\tau^*)^{\otimes 3}P_{12}P_{23}(x,y,z)=P(\tau^*(x),\tau^*\left(P(y,z)'\right))\otimes \tau^* \left(P(y,z)''\right).
    \end{equation*}
     Let us rewrite the expression on the right-hand side in a more compact form
     \begin{equation*}
         (12)(\tau^*)^{\otimes 3}P_{12}P_{23}(x,y,z)=P_{12}\left((\tau^*)^{\otimes 2}\right)_{23}P_{23}\left(\tau^*(x), y,z\right).
     \end{equation*}

     Now we apply Lemma \ref{prop 1} once again:
     \begin{equation*}
         (12)(\tau^*)^{\otimes 3}P_{12}P_{23}(x,y,z)=P_{12}(23)P_{23}\left(\tau^*(x), \tau^*(y),\tau^*(z)\right).
     \end{equation*}
     Finally, item (i) of Lemma \ref{lemma2} yields the desired result.
\end{proof}

\begin{lemma}\label{lem1}
 The following identities hold in $S\left(L(A,d)^{\phi,\tau}\right)$.

{\rm(i)} For any $a,b,c\in A$ and $X\in \operatorname{Mat}^*(d,\kk)^{\otimes 3}$ we have
\begin{equation}\label{f31}
    \Bigl\{\!\!\!\Bigr\{\phi(a),\left\{\!\!\left\{\phi(b),\phi(c)\right\}\!\!\right\}\Bigr\}\!\!\!\Bigr\}_{\mathbf L}\overset{3}{\otimes} X=-\Bigl\{\!\!\!\Bigr\{a,\left\{\!\!\left\{c,b\right\}\!\!\right\}\Bigr\}\!\!\!\Bigr\}_{\mathbf L}\overset{3}{\otimes}(12)(\tau^*)^{\otimes 3}\left(X\right).
\end{equation}

{\rm(ii)} For any $a,b\in A$ and $x,y\in\operatorname{Mat}^*(d,\kk)$ we have
\begin{equation}\label{f32}
    \{\!\!\{\phi(a),b\}\!\!\}\overset{2}{\otimes} P(\tau^*(x),y)=-\{\!\!\{\phi(b),a\}\!\!\}\overset{2}{\otimes} P(\tau^*(y),x).
\end{equation}
\end{lemma}

\begin{proof}
(i) For simplicity we denote the left-hand side of \eqref{f31} by LHS. By the basic relation \eqref{f30} we have
\begin{equation*}
\operatorname{LHS}=\phi^{\otimes 3}\Bigl\{\!\!\!\Bigr\{\phi(a),\left\{\!\!\left\{\phi(b),\phi(c)\right\}\!\!\right\}\Bigr\}\!\!\!\Bigr\}_{\mathbf L}\overset{3}{\otimes} (\tau^*)^{\otimes 3}(X).
\end{equation*}

Then we expand the right-hand side:
\begin{align*}
\operatorname{LHS}=\left(\phi^{\otimes 2}\Bigl\{\!\!\!\Bigl\{\phi(a),\{\!\!\{\phi(b),\phi(c)\}\!\!\}'\Bigr\}\!\!\!\Bigr\}\overset{2}{\otimes}(\tau^*)^{\otimes 2}(X'\otimes X'')\right)\Bigl(\phi\left(\{\!\!\{\phi(b),\phi(c)\}\!\!\}''\right)\otimes \tau^*(X''')\Bigr),
\end{align*}
where $X=X'\otimes X''\otimes X'''\in\operatorname{Mat}^*(d,\kk)^{\otimes 3}$, and use the fact that the double bracket $\{\!\!\{-,-\}\!\!\}$ is $\phi$-adapted:
\begin{align*}
\operatorname{LHS}=\left(\Bigl\{\!\!\!\Bigl\{a,\phi\left(\{\!\!\{\phi(b),\phi(c)\}\!\!\}'\right)\Bigr\}\!\!\!\Bigr\}^{\circ}\overset{2}{\otimes}(\tau^*)^{\otimes 2}(X'\otimes X'')\right)\Bigl(\phi\left(\{\!\!\{\phi(b),\phi(c)\}\!\!\}''\right)\otimes \tau^*(X''')\Bigr).
\end{align*}

Next, due to commutativity of the symmetric algebra we can rewrite the first factor:
\begin{equation*}
    \operatorname{LHS}=\left(\Bigl\{\!\!\!\Bigl\{a,\phi\left(\{\!\!\{\phi(b),\phi(c)\}\!\!\}'\right)\Bigr\}\!\!\!\Bigr\}\overset{2}{\otimes}(12)(\tau^*)^{\otimes 2}(X'\otimes X'')\right)\Bigl(\phi\left(\{\!\!\{\phi(b),\phi(c)\}\!\!\}''\right)\otimes \tau^*(X''')\Bigr).
\end{equation*}

Let us rewrite the result in a more compact form:
\begin{equation*}
    \operatorname{LHS}=\Bigl\{\!\!\!\Bigr\{a,\phi^{\otimes 2}\Bigl(\left\{\!\!\left\{\phi(b),\phi(c)\right\}\!\!\right\}\Bigr)\Bigr\}\!\!\!\Bigr\}_{\mathbf L}\overset{3}{\otimes}(12)(\tau^*)^{\otimes 3} X.
\end{equation*}

Finally, we use the fact that the double bracket is $\phi$-adapted once again and obtain \eqref{f31}.

\smallskip

(ii) By the basic relation \eqref{f30} we have
\begin{equation*}
     \{\!\!\{\phi(a),b\}\!\!\}\overset{2}{\otimes} P(\tau^*(x),y)=\phi^{\otimes 2}\left(\{\!\!\{\phi(a),b\}\!\!\}\right)\overset{2}{\otimes} (\tau^*)^{\otimes 2}P(\tau^*(x),y).
\end{equation*}

Next, we use the fact that the double bracket is $\phi$-adapted and skew symmetric:
\begin{equation*}
     \{\!\!\{\phi(a),b\}\!\!\}\overset{2}{\otimes} P(\tau^*(x),y)=-\{\!\!\{\phi(b),a\}\!\!\}\overset{2}{\otimes} (\tau^*)^{\otimes 2}P(\tau^*(x),y).
\end{equation*}

Finally, we use Lemma \ref{prop 1} and the fact that $P$ commutes with the transposition (12), see Lemma \ref{lemma2}, to obtain \eqref{f32}.
\end{proof}

\subsection{Proof of Proposition \ref{th1}}

Recall that by Remark \ref{rem2}, the bracket $\{-,-\}_{\phi,\tau}$ is well defined. Note that \eqref{f29} is skew symmetric due to item (ii) of  Lemma \ref{lem1}. Now we proceed to the proof of  \eqref{f28}.

By the definition \eqref{f29} we have
\begin{align*}
    \left\{f_2,f_3\right\}_{\phi,\tau}&=\{\!\!\{b,c\}\!\!\}\overset{2}{\otimes} P(y,z)+\{\!\!\{\phi(b),c\}\!\!\}\overset{2}{\otimes} P(\tau^*(y),z)\\
    &=\Bigl(\{\!\!\{b,c\}\!\!\}'\otimes P(y,z)'\Bigr)\Bigl(\{\!\!\{b,c\}\!\!\}''\otimes P(y,z)''\Bigr)\\
    &\phantom{aaaaaa}+\Bigl(\{\!\!\{\phi(b),c\}\!\!\}'\otimes P(\tau^*(y),z)'\Bigr)\Bigl(\{\!\!\{\phi(b),c\}\!\!\}''\otimes P(\tau^*(y),z)''\Bigr).
\end{align*}

Then we apply the Leibniz rule to each of the two resulting summands:
\begin{align*}
    \Bigl\{f_1,\{\Bigr. & \Bigl. f_2,f_3\}_{\phi,\tau}\Bigr\}_{\phi,\tau}\\
    &=\Bigl\{f_1,\{\!\!\{b,c\}\!\!\}'\otimes P(y,z)'\Bigr\}_{\phi,\tau}\Bigl(\{\!\!\{b,c\}\!\!\}''\otimes P(y,z)''\Bigr)\\
    &\phantom{aaaa}+\Bigl(\{\!\!\{b,c\}\!\!\}'\otimes P(y,z)'\Bigr)\Bigl\{f_1,\{\!\!\{b,c\}\!\!\}''\otimes P(y,z)''\Bigr\}_{\phi,\tau}\\
    &\phantom{aaaa}+\Bigl\{f_1,\{\!\!\{\phi(b),c\}\!\!\}'\otimes P(\tau^*(y),z)'\Bigr\}_{\phi,\tau}\Bigl(\{\!\!\{\phi(b),c\}\!\!\}''\otimes P(\tau^*(y),z)''\Bigr)\\
    &\phantom{aaaa}+\Bigl(\{\!\!\{\phi(b),c\}\!\!\}'\otimes P(\tau^*(y),z)'\Bigr)\Bigl\{f_1,\{\!\!\{\phi(b),c\}\!\!\}''\otimes P(\tau^*(y),z)''\Bigr\}_{\phi,\tau}.\\
\end{align*}
Expanding the bracket $\{-,-\}_{\phi,\tau}$ we obtain eight summands:
\begin{align*}
    \Bigl\{f_1\Bigr. & \Bigl.,\{ f_2,f_3\}_{\phi,\tau}\Bigr\}_{\phi,\tau}\\
    &=\left(\Bigl\{\!\!\!\Bigl\{a,\{\!\!\{b,c\}\!\!\}'\Bigr\}\!\!\!\Bigr\}\overset{2}{\otimes} P(x,P(y,z)')\right)\Bigl(\{\!\!\{b,c\}\!\!\}''\otimes P(y,z)''\Bigr)\\
    &\phantom{aaa}+\left(\Bigl\{\!\!\!\Bigl\{\phi(a),\{\!\!\{b,c\}\!\!\}'\Bigr\}\!\!\!\Bigr\}\overset{2}{\otimes} P(\tau^*(x),P(y,z)')\right)\Bigl(\{\!\!\{b,c\}\!\!\}''\otimes P(y,z)''\Bigr)\\
    &\phantom{aaa}+\Bigl(\{\!\!\{b,c\}\!\!\}'\otimes P(y,z)'\Bigr)\left(\Bigl\{\!\!\!\Bigl\{a,\{\!\!\{b,c\}\!\!\}''\Bigr\}\!\!\!\Bigr\}\overset{2}{\otimes} P(x,P(y,z)'')\right)\\
    &\phantom{aaa}+\Bigl(\{\!\!\{b,c\}\!\!\}'\otimes P(y,z)'\Bigr)\left(\Bigl\{\!\!\!\Bigl\{\phi(a),\{\!\!\{b,c\}\!\!\}''\Bigr\}\!\!\!\Bigr\}\overset{2}{\otimes} P(\tau^*(x),P(y,z)'')\right)\\
    &\phantom{aaa}+\left(\Bigl\{\!\!\!\Bigl\{a,\{\!\!\{\phi(b),c\}\!\!\}'\Bigr\}\!\!\!\Bigr\}\overset{2}{\otimes} P(x,P(\tau^*(y),z)')\right)\Bigl(\{\!\!\{\phi(b),c\}\!\!\}''\otimes P(\tau^*(y),z)''\Bigr)\\
    &\phantom{aaa}+\left(\Bigl\{\!\!\!\Bigl\{\phi(a),\{\!\!\{\phi(b),c\}\!\!\}'\Bigr\}\!\!\!\Bigr\}\overset{2}{\otimes} P(\tau^*(x),P(\tau^*(y),z)')\right)\Bigl(\{\!\!\{\phi(b),c\}\!\!\}''\otimes P(\tau^*(y),z)''\Bigr)\\
    &\phantom{aaa}+\Bigl(\{\!\!\{\phi(b),c\}\!\!\}'\otimes P(\tau^*(y),z)'\Bigr)\left(\Bigl\{\!\!\!\Bigl\{a,\{\!\!\{\phi(b),c\}\!\!\}''\Bigr\}\!\!\!\Bigr\}\overset{2}{\otimes} P(x,P(\tau^*(y),z)'')\right)\\
    &\phantom{aaa}+\Bigl(\{\!\!\{\phi(b),c\}\!\!\}'\otimes P(\tau^*(y),z)'\Bigr)\left(\Bigl\{\!\!\!\Bigl\{\phi(a),\{\!\!\{\phi(b),c\}\!\!\}''\Bigr\}\!\!\!\Bigr\}\overset{2}{\otimes} P(\tau^*(x),P(\tau^*(y),z)'')\right).
\end{align*}

Now we will perform a series of manipulations with these summands in order to rewrite them in a standard form, but we shall notice that the order of these summands in all the formulas below will remain unchanged.

Let us rewrite the resulting expression in a more compact form
\begin{align}\label{f37}
 \notag
    \Bigl\{f_1,\{f_2,f_3\}_{\phi,\tau}\Bigr\}_{\phi,\tau} &=\Bigl\{\!\!\!\Bigr\{a,\left\{\!\!\left\{b,c\right\}\!\!\right\}\Bigr\}\!\!\!\Bigr\}_{\mathbf L}\overset{3}{\otimes} P_{12}P_{23}(x,y,z)\\ \notag
    &\phantom{aaaa}+\Bigl\{\!\!\!\Bigr\{\phi(a),\left\{\!\!\left\{b,c\right\}\!\!\right\}\Bigr\}\!\!\!\Bigr\}_{\mathbf L}\overset{3}{\otimes} P_{12}P_{23}(\tau^*(x),y,z) \\ \notag
    &\phantom{aaaa}+\Bigl\{\!\!\!\Bigr\{a,\left\{\!\!\left\{b,c\right\}\!\!\right\}\Bigr\}\!\!\!\Bigr\}_{\mathbf R}\overset{3}{\otimes} P_{23}(12)P_{23}(x,y,z)\\ \notag
    &\phantom{aaaa}+\Bigl\{\!\!\!\Bigr\{\phi(a),\left\{\!\!\left\{b,c\right\}\!\!\right\}\Bigr\}\!\!\!\Bigr\}_{\mathbf R}\overset{3}{\otimes} P_{23}(12)P_{23}(\tau^*(x),y,z)\\ 
    &\phantom{aaaa}+\Bigl\{\!\!\!\Bigr\{a,\left\{\!\!\left\{\phi(b),c\right\}\!\!\right\}\Bigr\}\!\!\!\Bigr\}_{\mathbf L}\overset{3}{\otimes} P_{12}P_{23}(x,\tau^*(y),z)\\  \notag
    &\phantom{aaaa}+\Bigl\{\!\!\!\Bigr\{\phi(a),\left\{\!\!\left\{\phi(b),c\right\}\!\!\right\}\Bigr\}\!\!\!\Bigr\}_{\mathbf L}\overset{3}{\otimes} P_{12}P_{23}(\tau^*(x),\tau^*(y),z)\\ \notag
    &\phantom{aaaa}+\Bigl\{\!\!\!\Bigr\{a,\left\{\!\!\left\{\phi(b),c\right\}\!\!\right\}\Bigr\}\!\!\!\Bigr\}_{\mathbf R}\overset{3}{\otimes} P_{23}(12)P_{23}(x,\tau^*(y),z)\\ \notag
    &\phantom{aaaa}+\Bigl\{\!\!\!\Bigr\{\phi(a),\left\{\!\!\left\{\phi(b),c\right\}\!\!\right\}\Bigr\}\!\!\!\Bigr\}_{\mathbf R}\overset{3}{\otimes} P_{23}(12)P_{23}(\tau^*(x),\tau^*(y),z). \notag
\end{align}

Recall that $\{\!\!\{a,-\}\!\!\}_\Left$ is the action of $\{\!\!\{a,-\}\!\!\}$ on the left tensor factor of $A\otimes A$, i.e. $\{\!\!\{a,x\}\!\!\}_\Left=\{\!\!\{a,x'\}\!\!\}\otimes x''$ for any $x=x'\otimes x''\in \Atwo$. The expression $\{\!\!\{a,-\}\!\!\}_\Right$ is defined similarly, $\{\!\!\{a,x\}\!\!\}_\Right=x'\otimes \{\!\!\{a,x''\}\!\!\}$.

Next, we would like to bring all the terms to a standard form by which we mean that the result should contain only elements of the form $\Bigl\{\!\!\!\Bigr\{-,\left\{\!\!\left\{-,-\right\}\!\!\right\}\Bigr\}\!\!\!\Bigr\}_\Left$ but not $\Bigl\{\!\!\!\Bigr\{-,\left\{\!\!\left\{-,-\right\}\!\!\right\}\Bigr\}\!\!\!\Bigr\}_\Right$. Moreover, we wouldn't like to have two $\phi$'s inside brackets like we have now in the sixth and eighth summands. Let us perform the necessary computations step by step.

We replace $\mathbf{L}$ with $\mathbf{R}$ as follows. Firstly we use the skew-commutativity of the double bracket  
\begin{equation*}
    \Bigl\{\!\!\!\Bigr\{a,\left\{\!\!\left\{b,c\right\}\!\!\right\}\Bigr\}\!\!\!\Bigr\}_{\mathbf R}=-\Bigl\{\!\!\!\Bigr\{a,\left\{\!\!\left\{c,b\right\}\!\!\right\}^{\circ}\Bigr\}\!\!\!\Bigr\}_{\mathbf R}=-(123)\Bigl\{\!\!\!\Bigr\{a,\left\{\!\!\left\{c,b\right\}\!\!\right\}\Bigr\}\!\!\!\Bigr\}_{\mathbf L}.
\end{equation*}
Secondly we use commutativity of the symmetric algebra
\begin{equation}\label{f35}
    \Bigl\{\!\!\!\Bigr\{a,\left\{\!\!\left\{b,c\right\}\!\!\right\}\Bigr\}\!\!\!\Bigr\}_{\mathbf R}\overset{3}{\otimes} P_{23}(12)P_{23}(x,y,z)=-\Bigl\{\!\!\!\Bigr\{a,\left\{\!\!\left\{c,b\right\}\!\!\right\}\Bigr\}\!\!\!\Bigr\}_{\mathbf L}\overset{3}{\otimes} (123)^2P_{23}(12)P_{23}(x,y,z).
\end{equation}

Combining the following simple identities  
\begin{gather*}
    (123)^2P_{23}=P_{12}(123)^2\\
    (123)^2(12)=(23)
\end{gather*}
with the first part of Lemma \ref{lemma2} and \eqref{f35} we arrive at
\begin{equation}\label{f36}
    \Bigl\{\!\!\!\Bigr\{a,\left\{\!\!\left\{b,c\right\}\!\!\right\}\Bigr\}\!\!\!\Bigr\}_{\mathbf R}\overset{3}{\otimes} P_{23}(12)P_{23}(x,y,z)=-\Bigl\{\!\!\!\Bigr\{a,\left\{\!\!\left\{c,b\right\}\!\!\right\}\Bigr\}\!\!\!\Bigr\}_{\mathbf L}\overset{3}{\otimes} P_{12}P_{23}(23)(x,y,z).
\end{equation}

Now we can apply \eqref{f36} to various inputs to obtain desired expressions for the fourth, seventh, and eighth summands of \eqref{f37}:
\begin{multline}\label{f39}
    \Bigl\{\!\!\!\Bigr\{\phi(a),\left\{\!\!\left\{b,c\right\}\!\!\right\}\Bigr\}\!\!\!\Bigr\}_{\mathbf R}\overset{3}{\otimes} P_{23}(12)P_{23}(\tau^*(x),y,z)\\
    =-\Bigl\{\!\!\!\Bigr\{\phi(a),\left\{\!\!\left\{c,b\right\}\!\!\right\}\Bigr\}\!\!\!\Bigr\}_{\mathbf L}\overset{3}{\otimes} P_{12}P_{23}(23)(\tau^*(x),y,z)
\end{multline}
\begin{multline}\label{f40}
    \Bigl\{\!\!\!\Bigr\{a,\left\{\!\!\left\{\phi(b),c\right\}\!\!\right\}\Bigr\}\!\!\!\Bigr\}_{\mathbf R}\overset{3}{\otimes} P_{23}(12)P_{23}(x,\tau^*(y),z)\\
    =-\Bigl\{\!\!\!\Bigr\{a,\left\{\!\!\left\{c,\phi(b)\right\}\!\!\right\}\Bigr\}\!\!\!\Bigr\}_{\mathbf L}\overset{3}{\otimes} P_{12}P_{23}(23)(x,\tau^*(y),z)
\end{multline}
\begin{multline}\label{f38}
    \Bigl\{\!\!\!\Bigr\{\phi(a),\left\{\!\!\left\{\phi(b),c\right\}\!\!\right\}\Bigr\}\!\!\!\Bigr\}_{\mathbf R}\overset{3}{\otimes} P_{23}(12)P_{23}(\tau^*(x),\tau^*(y),z)\\
    =-\Bigl\{\!\!\!\Bigr\{\phi(a),\left\{\!\!\left\{c,\phi(b)\right\}\!\!\right\}\Bigr\}\!\!\!\Bigr\}_{\mathbf L}\overset{3}{\otimes} P_{12}P_{23}(23)(\tau^*(x),\tau^*(y),z). 
\end{multline}
   
Now we would like to rewrite the summands containing two $\phi$'s in terms of similar expressions with only one $\phi$ with the help of Lemma \ref{lem1}. We do this for the sixth term of \eqref{f37} 
\begin{multline*}
    \Bigl\{\!\!\!\Bigr\{\phi(a),\left\{\!\!\left\{\phi(b),c\right\}\!\!\right\}\Bigr\}\!\!\!\Bigr\}_{\mathbf L}\overset{3}{\otimes} P_{12}P_{23}(\tau^*(x),\tau^*(y),z)\\
    =-\Bigl\{\!\!\!\Bigr\{a,\left\{\!\!\left\{\phi(c),b\right\}\!\!\right\}\Bigr\}\!\!\!\Bigr\}_{\mathbf L}\overset{3}{\otimes} (12)(\tau^*)^{\otimes 3}P_{12}P_{23}(\tau^*(x),\tau^*(y),z)
\end{multline*}
and for the right-hand side of \eqref{f38}
\begin{multline*}
   \Bigl\{\!\!\!\Bigr\{\phi(a),\left\{\!\!\left\{\phi(b),c\right\}\!\!\right\}\Bigr\}\!\!\!\Bigr\}_{\mathbf R}\overset{3}{\otimes} P_{23}(12)P_{23}(\tau^*(x),\tau^*(y),z)\\
    =\Bigl\{\!\!\!\Bigr\{a,\left\{\!\!\left\{b,\phi(c)\right\}\!\!\right\}\Bigr\}\!\!\!\Bigr\}_{\mathbf L}\overset{3}{\otimes} (12)(\tau^*)^{\otimes 3}P_{12}P_{23}(23)(\tau^*(x),\tau^*(y),z)
\end{multline*}

Then we apply Corollary \ref{corollary1} to the right-hand sides of the above identities to obtain 
\begin{multline}\label{f41}
    \Bigl\{\!\!\!\Bigr\{\phi(a),\left\{\!\!\left\{\phi(b),c\right\}\!\!\right\}\Bigr\}\!\!\!\Bigr\}_{\mathbf L}\overset{3}{\otimes} P_{12}P_{23}(\tau^*(x),\tau^*(y),z)\\
    =-\Bigl\{\!\!\!\Bigr\{a,\left\{\!\!\left\{\phi(c),b\right\}\!\!\right\}\Bigr\}\!\!\!\Bigr\}_{\mathbf L}\overset{3}{\otimes} P_{12}P_{23}(23)(x,y,\tau^*(z))
\end{multline}
and
\begin{multline}\label{f42}
    \Bigl\{\!\!\!\Bigr\{\phi(a),\left\{\!\!\left\{\phi(b),c\right\}\!\!\right\}\Bigr\}\!\!\!\Bigr\}_{\mathbf R}\overset{3}{\otimes} P_{23}(12)P_{23}(\tau^*(x),\tau^*(y),z)\\
    =\Bigl\{\!\!\!\Bigr\{a,\left\{\!\!\left\{b,\phi(c)\right\}\!\!\right\}\Bigr\}\!\!\!\Bigr\}_{\mathbf L}\overset{3}{\otimes} P_{12}P_{23}(x,y,\tau^*(z)).
\end{multline}

Finally, we can combine \eqref{f37} with \eqref{f36}, \eqref{f39}, \eqref{f40}, \eqref{f41}, and \eqref{f42}
\begin{align}\label{f43}
 \notag
 \Bigl\{f_1,\{f_2,f_3\}_{\phi,\tau}\Bigr\}_{\phi,\tau}&=\Bigl\{\!\!\!\Bigr\{a,\left\{\!\!\left\{b,c\right\}\!\!\right\}\Bigr\}\!\!\!\Bigr\}_{\mathbf L}\overset{3}{\otimes} P_{12}P_{23}(x,y,z)\\ \notag
    &\phantom{aaaaa}+\Bigl\{\!\!\!\Bigr\{\phi(a),\left\{\!\!\left\{b,c\right\}\!\!\right\}\Bigr\}\!\!\!\Bigr\}_{\mathbf L}\overset{3}{\otimes} P_{12}P_{23}(\tau^*(x),y,z)\\ \notag
    &\phantom{aaaaa}-\Bigl\{\!\!\!\Bigr\{a,\left\{\!\!\left\{c,b\right\}\!\!\right\}\Bigr\}\!\!\!\Bigr\}_{\mathbf L}\overset{3}{\otimes} P_{12}P_{23}(23)(x,y,z)\\ \notag
    &\phantom{aaaaa}-\Bigl\{\!\!\!\Bigr\{\phi(a),\left\{\!\!\left\{c,b\right\}\!\!\right\}\Bigr\}\!\!\!\Bigr\}_{\mathbf L}\overset{3}{\otimes} P_{12}P_{23}(23)(\tau^*(x),y,z)\\ 
    &\phantom{aaaaa}+\Bigl\{\!\!\!\Bigr\{a,\left\{\!\!\left\{\phi(b),c\right\}\!\!\right\}\Bigr\}\!\!\!\Bigr\}_{\mathbf L}\overset{3}{\otimes} P_{12}P_{23}(x,\tau^*(y),z)\\ \notag
    &\phantom{aaaaa}-\Bigl\{\!\!\!\Bigr\{a,\left\{\!\!\left\{\phi(c),b\right\}\!\!\right\}\Bigr\}\!\!\!\Bigr\}_{\mathbf L}\overset{3}{\otimes} P_{12}P_{23}(23)(x,y,\tau^*(z))\\ \notag
    &\phantom{aaaaa}-\Bigl\{\!\!\!\Bigr\{a,\left\{\!\!\left\{c,\phi(b)\right\}\!\!\right\}\Bigr\}\!\!\!\Bigr\}_{\mathbf L}\overset{3}{\otimes} P_{12}P_{23}(23)(x,\tau^*(y),z)\\ \notag
    &\phantom{aaaaa}+\Bigl\{\!\!\!\Bigr\{a,\left\{\!\!\left\{b,\phi(c)\right\}\!\!\right\}\Bigr\}\!\!\!\Bigr\}_{\mathbf L}\overset{3}{\otimes} P_{12}P_{23}(x,y,\tau^*(z)). \notag
\end{align}

Now we can apply \eqref{f43} to various inputs in order to obtain similar expressions for $\Bigl\{f_2,\{f_3,f_1\}_{\phi,\tau}\Bigr\}_{\phi,\tau}$ and $\Bigl\{f_3,\{f_1,f_2\}_{\phi,\tau}\Bigr\}_{\phi,\tau}$. Notice the permutations $(123)$ and $(123)^2$ inside the right tensor factor in formulas below:
\begin{align*}
    \Bigl\{f_2,\{f_3,f_1\}_{\phi,\tau}\Bigr\}_{\phi,\tau}&=\Bigl\{\!\!\!\Bigr\{b,\left\{\!\!\left\{c,a\right\}\!\!\right\}\Bigr\}\!\!\!\Bigr\}_{\mathbf L}\overset{3}{\otimes} P_{12}P_{23}(123)^2(x,y,z)\\
    &\phantom{aaaaa}+\Bigl\{\!\!\!\Bigr\{\phi(b),\left\{\!\!\left\{c,a\right\}\!\!\right\}\Bigr\}\!\!\!\Bigr\}_{\mathbf L}\overset{3}{\otimes} P_{12}P_{23}(123)^2(x,\tau^*(y),z) \\
    &\phantom{aaaaa}-\Bigl\{\!\!\!\Bigr\{b,\left\{\!\!\left\{a,c\right\}\!\!\right\}\Bigr\}\!\!\!\Bigr\}_{\mathbf L}\overset{3}{\otimes} P_{12}P_{23}(23)(123)^2(x,y,z)\\
    &\phantom{aaaaa}-\Bigl\{\!\!\!\Bigr\{\phi(b),\left\{\!\!\left\{a,c\right\}\!\!\right\}\Bigr\}\!\!\!\Bigr\}_{\mathbf L}\overset{3}{\otimes} P_{12}P_{23}(23)(123)^2(x,\tau^*(y),z)\\
    &\phantom{aaaaa}+\Bigl\{\!\!\!\Bigr\{b,\left\{\!\!\left\{\phi(c),a\right\}\!\!\right\}\Bigr\}\!\!\!\Bigr\}_{\mathbf L}\overset{3}{\otimes} P_{12}P_{23}(123)^2(x,y,\tau^*(z))\\
    &\phantom{aaaaa}-\Bigl\{\!\!\!\Bigr\{b,\left\{\!\!\left\{\phi(a),c\right\}\!\!\right\}\Bigr\}\!\!\!\Bigr\}_{\mathbf L}\overset{3}{\otimes} P_{12}P_{23}(23)(123)^2(\tau^*(x),y,z)\\
    &\phantom{aaaaa}-\Bigl\{\!\!\!\Bigr\{b,\left\{\!\!\left\{a,\phi(c)\right\}\!\!\right\}\Bigr\}\!\!\!\Bigr\}_{\mathbf L}\overset{3}{\otimes} P_{12}P_{23}(23)(123)^2(x,y,\tau^*(z))\\
    &\phantom{aaaaa}+\Bigl\{\!\!\!\Bigr\{b,\left\{\!\!\left\{c,\phi(a)\right\}\!\!\right\}\Bigr\}\!\!\!\Bigr\}_{\mathbf L}\overset{3}{\otimes} P_{12}P_{23}(123)^2(\tau^*(x),y,z)
\end{align*}
and
\begin{align*}
    \Bigl\{f_3,\{f_1,f_2\}_{\phi,\tau}\Bigr\}_{\phi,\tau}&=\Bigl\{\!\!\!\Bigr\{c,\left\{\!\!\left\{a,b\right\}\!\!\right\}\Bigr\}\!\!\!\Bigr\}_{\mathbf L}\overset{3}{\otimes} P_{12}P_{23}(123)(x,y,z)\\
    &\phantom{aaaaa}+\Bigl\{\!\!\!\Bigr\{\phi(c),\left\{\!\!\left\{a,b\right\}\!\!\right\}\Bigr\}\!\!\!\Bigr\}_{\mathbf L}\overset{3}{\otimes} P_{12}P_{23}(123)(x,y,\tau^*(z)) \\
    &\phantom{aaaaa}-\Bigl\{\!\!\!\Bigr\{c,\left\{\!\!\left\{b,a\right\}\!\!\right\}\Bigr\}\!\!\!\Bigr\}_{\mathbf L}\overset{3}{\otimes} P_{12}P_{23}(23)(123)(x,y,z)\\
    &\phantom{aaaaa}-\Bigl\{\!\!\!\Bigr\{\phi(c),\left\{\!\!\left\{b,a\right\}\!\!\right\}\Bigr\}\!\!\!\Bigr\}_{\mathbf L}\overset{3}{\otimes} P_{12}P_{23}(23)(123)(x,y,\tau^*(z))\\
    &\phantom{aaaaa}+\Bigl\{\!\!\!\Bigr\{c,\left\{\!\!\left\{\phi(a),b\right\}\!\!\right\}\Bigr\}\!\!\!\Bigr\}_{\mathbf L}\overset{3}{\otimes} P_{12}P_{23}(123)(\tau^*(x),y,z)\\
    &\phantom{aaaaa}-\Bigl\{\!\!\!\Bigr\{c,\left\{\!\!\left\{\phi(b),a\right\}\!\!\right\}\Bigr\}\!\!\!\Bigr\}_{\mathbf L}\overset{3}{\otimes} P_{12}P_{23}(23)(123)(x,\tau^*(y),z)\\
    &\phantom{aaaaa}-\Bigl\{\!\!\!\Bigr\{c,\left\{\!\!\left\{b,\phi(a)\right\}\!\!\right\}\Bigr\}\!\!\!\Bigr\}_{\mathbf L}\overset{3}{\otimes} P_{12}P_{23}(23)(123)(\tau^*(x),y,z)\\
    &\phantom{aaaaa}+\Bigl\{\!\!\!\Bigr\{c,\left\{\!\!\left\{a,\phi(b)\right\}\!\!\right\}\Bigr\}\!\!\!\Bigr\}_{\mathbf L}\overset{3}{\otimes} P_{12}P_{23}(123)(x,\tau^*(y),z).
\end{align*}

Next, we would like to move the permutations $(123)$ and $(123)^2$ to the first tensor factor. Let's do this in two steps. Firstly, we use the second part of Lemma \ref{lemma2} and the following simple identity if needed
\begin{equation*}
    (23)(123)=(123)^2(23)
\end{equation*}
 to move the permutations to the leftmost position of the second tensor factor. Secondly, we use commutativity of the symmetric algebra to move the permutations to the first tensor factor. The results are as follows
\begin{align}\label{f46}
 \notag
    \Bigl\{f_2,\{f_3,f_1\}_{\phi,\tau}\Bigr\}_{\phi,\tau}&=(123)\Bigl\{\!\!\!\Bigr\{b,\left\{\!\!\left\{c,a\right\}\!\!\right\}\Bigr\}\!\!\!\Bigr\}_{\mathbf L}\overset{3}{\otimes} P_{12}P_{23}(x,y,z)\\ \notag
    &\phantom{aaaaa}+(123)\Bigl\{\!\!\!\Bigr\{\phi(b),\left\{\!\!\left\{c,a\right\}\!\!\right\}\Bigr\}\!\!\!\Bigr\}_{\mathbf L}\overset{3}{\otimes} P_{12}P_{23}(x,\tau^*(y),z) \\ \notag
    &\phantom{aaaaa}-(123)^2\Bigl\{\!\!\!\Bigr\{b,\left\{\!\!\left\{a,c\right\}\!\!\right\}\Bigr\}\!\!\!\Bigr\}_{\mathbf L}\overset{3}{\otimes} P_{12}P_{23}(23)(x,y,z)\\ \notag
    &\phantom{aaaaa}-(123)^2\Bigl\{\!\!\!\Bigr\{\phi(b),\left\{\!\!\left\{a,c\right\}\!\!\right\}\Bigr\}\!\!\!\Bigr\}_{\mathbf L}\overset{3}{\otimes} P_{12}P_{23}(23)(x,\tau^*(y),z)\\
    &\phantom{aaaaa}+(123)\Bigl\{\!\!\!\Bigr\{b,\left\{\!\!\left\{\phi(c),a\right\}\!\!\right\}\Bigr\}\!\!\!\Bigr\}_{\mathbf L}\overset{3}{\otimes} P_{12}P_{23}(x,y,\tau^*(z))\\ \notag
    &\phantom{aaaaa}-(123)^2\Bigl\{\!\!\!\Bigr\{b,\left\{\!\!\left\{\phi(a),c\right\}\!\!\right\}\Bigr\}\!\!\!\Bigr\}_{\mathbf L}\overset{3}{\otimes} P_{12}P_{23}(23)(\tau^*(x),y,z)\\ \notag
    &\phantom{aaaaa}-(123)^2\Bigl\{\!\!\!\Bigr\{b,\left\{\!\!\left\{a,\phi(c)\right\}\!\!\right\}\Bigr\}\!\!\!\Bigr\}_{\mathbf L}\overset{3}{\otimes} P_{12}P_{23}(23)(x,y,\tau^*(z))\\ \notag
    &\phantom{aaaaa}+(123)\Bigl\{\!\!\!\Bigr\{b,\left\{\!\!\left\{c,\phi(a)\right\}\!\!\right\}\Bigr\}\!\!\!\Bigr\}_{\mathbf L}\overset{3}{\otimes} P_{12}P_{23}(\tau^*(x),y,z) \notag
\end{align}
and
\begin{align}\label{f47}
 \notag
    \Bigl\{f_3,\{f_1,f_2\}_{\phi,\tau}\Bigr\}_{\phi,\tau}&=(123)^2\Bigl\{\!\!\!\Bigr\{c,\left\{\!\!\left\{a,b\right\}\!\!\right\}\Bigr\}\!\!\!\Bigr\}_{\mathbf L}\overset{3}{\otimes} P_{12}P_{23}(x,y,z)\\ \notag
    &\phantom{aaaaa}+(123)^2\Bigl\{\!\!\!\Bigr\{\phi(c),\left\{\!\!\left\{a,b\right\}\!\!\right\}\Bigr\}\!\!\!\Bigr\}_{\mathbf L}\overset{3}{\otimes} P_{12}P_{23}(x,y,\tau^*(z)) \\ \notag
    &\phantom{aaaaa}-(123)\Bigl\{\!\!\!\Bigr\{c,\left\{\!\!\left\{b,a\right\}\!\!\right\}\Bigr\}\!\!\!\Bigr\}_{\mathbf L}\overset{3}{\otimes} P_{12}P_{23}(23)(x,y,z)\\ \notag
    &\phantom{aaaaa}-(123)\Bigl\{\!\!\!\Bigr\{\phi(c),\left\{\!\!\left\{b,a\right\}\!\!\right\}\Bigr\}\!\!\!\Bigr\}_{\mathbf L}\overset{3}{\otimes} P_{12}P_{23}(23)(x,y,\tau^*(z))\\
    &\phantom{aaaaa}+(123)^2\Bigl\{\!\!\!\Bigr\{c,\left\{\!\!\left\{\phi(a),b\right\}\!\!\right\}\Bigr\}\!\!\!\Bigr\}_{\mathbf L}\overset{3}{\otimes} P_{12}P_{23}(\tau^*(x),y,z)\\ \notag
    &\phantom{aaaaa}-(123)\Bigl\{\!\!\!\Bigr\{c,\left\{\!\!\left\{\phi(b),a\right\}\!\!\right\}\Bigr\}\!\!\!\Bigr\}_{\mathbf L}\overset{3}{\otimes} P_{12}P_{23}(23)(x,\tau^*(y),z)\\ \notag
    &\phantom{aaaaa}-(123)\Bigl\{\!\!\!\Bigr\{c,\left\{\!\!\left\{b,\phi(a)\right\}\!\!\right\}\Bigr\}\!\!\!\Bigr\}_{\mathbf L}\overset{3}{\otimes} P_{12}P_{23}(23)(\tau^*(x),y,z)\\ \notag
    &\phantom{aaaaa}+(123)^2\Bigl\{\!\!\!\Bigr\{c,\left\{\!\!\left\{a,\phi(b)\right\}\!\!\right\}\Bigr\}\!\!\!\Bigr\}_{\mathbf L}\overset{3}{\otimes} P_{12}P_{23}(x,\tau^*(y),z). \notag
\end{align}

Finally, we can sum \eqref{f43}, \eqref{f46}, and \eqref{f47} to obtain the desired expression for the left-hand side of \eqref{f28}
\begin{align}\label{f48}
 \notag
    &\Bigl\{f_1,\{f_2,f_3\}_{\phi,\tau}\Bigr\}_{\phi,\tau}+\Bigl\{f_2,\{f_3,f_1\}_{\phi,\tau}\Bigr\}_{\phi,\tau}+\Bigl\{f_3,\{f_1,f_2\}_{\phi,\tau}\Bigr\}_{\phi,\tau}\\ \notag
    &\phantom{aaaaaaaaaaai}=\{\!\!\{a,b,c\}\!\!\}\overset{3}{\otimes} P_{12}P_{23}\left(x,y,z\right)-\{\!\!\{a,c,b\}\!\!\}\overset{3}{\otimes}P_{12}P_{23}(23)\left(x,y,z\right)\\ \notag
    &+\{\!\!\{\phi(a),b,c\}\!\!\}\overset{3}{\otimes} P_{12}P_{23}\left(\tau^*(x),y,z\right)-\{\!\!\{\phi(a),c,b\}\!\!\}\overset{3}{\otimes}P_{12}P_{23}(23)\left(\tau^*(x),y,z\right)\\
    &+\{\!\!\{a,\phi(b),c\}\!\!\}\overset{3}{\otimes} P_{12}P_{23}\left(x,\tau^*(y),z\right)-\{\!\!\{a,c,\phi(b)\}\!\!\}\overset{3}{\otimes}P_{12}P_{23}(23)\left(x,\tau^*(y),z\right)\\ \notag
    &+\{\!\!\{a,b,\phi(c)\}\!\!\}\overset{3}{\otimes} P_{12}P_{23}\left(x,y,\tau^*(z)\right)-\{\!\!\{a,\phi(c),b\}\!\!\}\overset{3}{\otimes}P_{12}P_{23}(23)\left(x,y,\tau^*(z)\right)\\ \notag
    &\phantom{aaaaaaaaaaa}=\VdB(a,b,c;x,y,z)+\VdB(\phi(a),b,c;\tau^*(x),y,z)\\ \notag
    &\phantom{aaaaaaaaaaa}+\VdB(a,\phi(b),c;x,\tau^*(y),z)+\VdB(a,b,\phi(c);x,y,\tau^*(z)).
\end{align}

The second equality in \eqref{f48} is trivial -- simply compare the two expressions term-wise. Let us explain the first equality. For this we indicate how to group summands from \eqref{f43}, \eqref{f46}, and \eqref{f47} to obtain \eqref{f48}. Below we present a list of triples of natural numbers. The position in the list corresponds to a summand from \eqref{f48}, which is comprised of three summands from \eqref{f43}, \eqref{f46}, and \eqref{f47}. The three numbers correspond to these summands. For instance, the third item of the list says that we sum the second summand from \eqref{f43}, the eighth summand from \eqref{f46}, and the fifth summand from \eqref{f47} to obtain the third summand in \eqref{f48}.
\begin{align*}
    &&1)\ (1,1,1)\ \ \ \ &&2)\ (3,3,3)\ \ \ \ &&3)\ (2,8,5)\ \ \ \ &&4)\ (4,6,7)\\
    &&5)\ (5,2,8)\ \ \ \ &&6)\ (7,4,6)\ \ \ \ &&7)\ (8,5,2)\ \ \ \ &&8)\ (6,7,4)
\end{align*}

This completes the  proof of Proposition \ref{th1}.

\subsection{Main theorem}

Now we are in a position to formulate and prove our main result.

\begin{theorem}\label{thm3.A}
Let $(A,\phi)$ be an involutive $\kk$-algebra and $\{\!\!\{-,-\}\!\!\}$ be a double Poisson bracket on $A$ which is $\phi$-adapted. Let $d$ be a positive integer and $\tau$ be an involution on the matrix algebra $\operatorname{Mat}(d,\kk)$. Then the formula \eqref{f52}, 
\begin{equation*}
\left\{(a|x),(b|y)\right\}_{\phi,\tau}=\Bigl(\{\!\!\{a,b\}\!\!\}\mathlarger{|} P(x,y)\Bigr)+\Bigl(\{\!\!\{\phi(a),b\}\!\!\}\mathlarger{|} P(\tau^*(x),y)\Bigr),
\end{equation*}
where $a,b\in A$ and $x,y\in\operatorname{Mat}^*(d,\kk)$, gives rise to a Poisson bracket on the commutative algebra $\OO(A,d)^{\phi,\tau}$.
\end{theorem}

\begin{proof}
We argue as in the proof of Proposition \ref{th1.B}.  Denote by $I$ the ideal of $S\left(L(A,d)^{\phi,\tau}\right)$ generated by the elements 
$$
(bc)\otimes y-\left(b\otimes \Delta(y)'\right)\left(c\otimes \Delta(y)''\right) \quad \text{and} \quad 1\otimes y-\eps(y),
$$
where $b, c\in A$ and $y\in \Mat^*(d,\kk)$. 

Next, set $\Sym(A,d)^{\phi,\tau}:=S\left(L(A,d)^{\phi,\tau}\right)$. By virtue of Proposition \ref{th1}, it suffices to check that $I$ is a Poisson ideal, that is $\{X,Y\}_{\phi,\tau}\in I$ for any $X\in \Sym(A,d)^{\phi,\tau}$ and $Y\in I$.  

It suffices to check this on the generators. For $Y=1\otimes y-\eps(y)$ this is clear. Next, we take arbitrary $a,b,c\in A$ and $x,y\in\operatorname{Mat}^*(d,\kk)$ and check that
    \begin{equation*}
         \Bigl\{a\otimes x, (bc)\otimes y-\left(b\otimes \Delta(y)'\right)\left(c\otimes \Delta(y)''\right)\Bigr\}_{\phi,\tau}\in I.
    \end{equation*}
    
By the very definition of the bracket we have
\begin{equation}\label{f53}
    \{a\otimes x, (bc)\otimes y\}_{\phi,\tau}=\{\!\!\{a,bc\}\!\!\}\overset{2}{\otimes} P(x,y)+\{\!\!\{\phi(a),bc\}\!\!\}\overset{2}{\otimes} P(\tau^*(x),y).
\end{equation}
On the other hand, by the Leibniz rule we have
\begin{multline*}
    \Bigl\{a\otimes x,\left(b\otimes \Delta(y)'\right)\left(c\otimes \Delta(y)''\right)\Bigr\}_{\phi,\tau}=\Bigl\{a\otimes x,b\otimes \Delta(y)'\Bigr\}_{\phi,\tau}\left(c\otimes \Delta(y)''\right)\\
    +\left(b\otimes \Delta(y)'\right)\Bigl\{a\otimes x,c\otimes \Delta(y)''\Bigr\}_{\phi,\tau}.
\end{multline*}
Then expanding the bracket $\{-,-\}_{\phi,\tau}$, we obtain
\begin{multline}\label{f54}
    \Bigl\{a\otimes x,\left(b\otimes \Delta(y)'\right)\left(c\otimes \Delta(y)''\right)\Bigr\}_{\phi,\tau}=\Bigl(\{\!\!\{a,b\}\!\!\}\overset{2}{\otimes} P(x,\Delta(y)')\Bigr)\left(c\otimes \Delta(y)''\right)\\
    +\Bigl(\{\!\!\{\phi(a),b\}\!\!\}\overset{2}{\otimes} P(\tau^*(x),\Delta(y)')\Bigr)\left(c\otimes \Delta(y)''\right)+\left(b\otimes \Delta(y)'\right)\Bigl(\{\!\!\{a,c\}\!\!\}\overset{2}{\otimes} P(x,\Delta(y)'')\Bigr)\\
    +\left(b\otimes \Delta(y)'\right)\Bigl(\{\!\!\{\phi(a),c\}\!\!\}\overset{2}{\otimes} P(\tau^*(x),\Delta(y)'')\Bigr).
\end{multline}
Next, we compare identities \eqref{f53} and \eqref{f54}. Upon comparison, we observe that certain summands in both formulas do not include $\phi$ and $\tau$. In light of this, we need to distinguish and separate the summands that contain $\phi$ and $\tau$ from those that do not. The summands without $\phi$ and $\tau$ lead to the following expression
\begin{multline}\label{f55}
    \{\!\!\{a,bc\}\!\!\}\overset{2}{\otimes} P(x,y)-\Bigl(\{\!\!\{a,b\}\!\!\}\overset{2}{\otimes} P(x,\Delta(y)')\Bigr)\left(c\otimes \Delta(y)''\right)\\
    -\left(b\otimes \Delta(y)'\right)\Bigl(\{\!\!\{a,c\}\!\!\}\overset{2}{\otimes} P(x,\Delta(y)'')\Bigr),
\end{multline}
which definitely belongs to the ideal $I$ by the proof of Proposition \ref{th1.B}.

Let us denote expression \eqref{f55} by $D(a,b,c;x,y)$.   

Now we turn to the summands that do contain $\phi$ and $\tau$. They lead to the following expression 
\begin{multline*}
    \{\!\!\{\phi(a),bc\}\!\!\}\overset{2}{\otimes} P(\tau^*(x),y)-\Bigl(\{\!\!\{\phi(a),b\}\!\!\}\overset{2}{\otimes} P(\tau^*(x),\Delta(y)')\Bigr)\left(c\otimes \Delta(y)''\right)\\
    -\left(b\otimes \Delta(y)'\right)\Bigl(\{\!\!\{\phi(a),c\}\!\!\}\overset{2}{\otimes} P(\tau^*(x),\Delta(y)'')\Bigr),
\end{multline*}
which is $D(\phi(a),b,c;\tau^*(x),y)$, hence belongs to the ideal $I$ and the proof is complete.
\end{proof}

\subsection{Poisson symmetries}

The general linear group $\GL(d,\kk)$ acts by conjugations on the matrix space $\Mat(d,\kk)$ and hence on the dual space $\Mat^*(d,\kk)$, too. Let us denote the latter action by $\operatorname{Ad}^*$. It gives rise to a natural action of $\GL(d,\kk)$ on $\OO(A,d)$: in our coordinate-free notation it is written simply as  
\begin{equation*}
g: (a|x)\mapsto (a|\operatorname{Ad}^*_g(x)), \qquad g\in \GL(d,\kk), \quad a\in A, \quad x\in\Mat^*(d,\kk).
\end{equation*}
This action obviously agrees with the natural action of $\GL(d,\kk)$ on the representation space $\Rep(A,d)$. 

Moreover, this action also preserves the bracket $\{-,-\}$ on $\OO(A,d)$ defined by \eqref{eq2.E}. Indeed, this follows from the fact that the operator $P: \Mat^*(d,\kk)^{\otimes 2}\to \Mat^*(d,\kk)^{\otimes 2}$ is invariant under $\operatorname{Ad}^*\otimes \operatorname{Ad}^*$, which in turn follows from Remark \ref{rem3} and the evident invariance property of the element $\sigma\in \Mat(d,\kk)^{\otimes 2}$ under $\operatorname{Ad}^{\otimes 2}$. 

We want to notice that similar claims hold in the twisted case as well. Namely, let $\tau$ be an involution on $\Mat(d,\kk)$ and  set  $G(d,\tau):=\{g\in \GL(d,\kk)\mid \tau(g)=g^{-1}\}$; this is either the orthogonal or the symplectic group, depending on the type of $\tau$.  

\begin{proposition}
Let $(A,\phi)$ be an involutive algebra and\/ $\bl-,-\br$ be a $\phi$-adapted double Poisson bracket on $A$. There is a natural action of the group $G(d,\tau)$ by automorphisms of the twisted coordinate ring $\OO(A,d)^{\phi,\tau}$ preserving the Poisson bracket constructed in Theorem \ref{thm3.A}.
\end{proposition}

\begin{proof}
Recall (Definition \ref{def2}) that $\OO(A,d)^{\phi,\tau}$ is the quotient of $\OO(A,d)$ by the ideal generated by the elements \eqref{eq3.F}. Since the set of these elements is $G(d,\tau)$-invariant, the action of $G(d,\tau)$ on $\OO(A,d)$ descends to $\OO(A,d)^{\phi,\tau}$. Note that the resulting action agrees with the natural action of $G(d,\tau)$ on the involutive representation space $\Rep(A,d)^{\phi,\tau}$. 

It remains to check the $G(d,\tau)$-invariance of the Poisson bracket \eqref{f52}. The first summand on the right-hand of \eqref{f52} is invariant, because it has the same form as the bracket on $\OO(A,d)$, which is invariant under the action of the larger group $\GL(d,\kk)$. This in turn implies that the second summand is invariant, too, because the map $x\to \tau^*(x)$ is $G(d,\tau)$-equivariant.     
\end{proof}

\section{Examples of $\phi$-adapted double Poisson brackets}\label{sect4}

\subsection{Preliminaries}
Let $(A,\phi)$ be an involutive algebra. Recall that a double Poisson bracket $\bl -,-\br$ on $A$ is said to be \emph{$\phi$-adapted} if
\begin{equation}\label{eq3.A} 
\phi^{\otimes2}(\bl a, b\br)=\bl \phi(a),\phi(b)\br^\circ, \qquad a,b\in A.
\end{equation} 

\begin{lemma}\label{lemma3.A}
Let $\{a_i: i\in I\}$ be a system of generators of $A$ and suppose that the relation \eqref{eq3.A} holds for all pairs of the form $a=a_i$, $b=a_j$, where $i,j\in I$. Then $\bl-,-\br$ is $\phi$-adapted.  
\end{lemma}

\begin{proof}

By induction, it suffices to show that if the relation \eqref{eq3.A} holds for some pairs $(a, b)$ and $(a,c)$, then it holds for the pair $(a,bc)$.

The verification is straightforward. To simplify the formulas we write
$$
\bl a,b\br=x'\otimes x'', \qquad \bl a,c\br=y'\otimes y''.
$$
We have 
$$
\bl a,bc\br=b\bl a,c\br +\bl a,b\br c=by'\otimes y''+ x'\otimes x'' c.
$$
It follows
\begin{equation}\label{eq3.B}
\phi^{\otimes 2}(\bl a,bc\br)=\phi(y')\phi(b)\otimes \phi(y'')+ \phi(x')\otimes \phi(c)\phi(x'').
\end{equation}
This is the left-hand side of the relation \eqref{eq3.A} for the pair $(a,bc)$.

Let us turn to the right-hand side. We have 
\begin{equation}\label{eq3.C}
\bl \phi(a), \phi(bc)\br=\bl \phi(a), \phi(c)\phi(b)\br\\
=\phi(c)\bl \phi(a),\phi(b)\br + \bl \phi(a), \phi(c)\br \phi(b).
\end{equation}
Now we use the assumption that the relation \eqref{eq3.A} holds for the pairs $(a,b)$ and $(a,c)$. We write this as
$$
\bl \phi(a),\phi(b)\br=\left(\phi^{\otimes2}(\bl a, b\br)\right)^\circ, \qquad \bl \phi(a),\phi(c)\br=\left(\phi^{\otimes2}(\bl a, c\br)\right)^\circ.
$$
It follows that \eqref{eq3.C} equals
$$
\phi(c)\phi(x'')\otimes \phi(x')+ \phi(y'')\otimes \phi(y')\phi(b).
$$
Finally we have to swap the tensor factors, as on the right-hand side of \eqref{eq3.A}. After that the result coincides with \eqref{eq3.B}, as desired.
\end{proof}

In the rest of the section we suppose that $A$ is the free unital algebra with $L\ge1$ generators $\al_1,\dots,\al_L$. 

\subsection{First example: linear $\phi$-adapted brackets}\label{sect4.2}

Following Pichereau and Van de Weyer \cite{pich} we say that a double Poisson bracket on $\kk\langle \al_1,\dots,\al_L\rangle$ is \emph{linear} if its values on the generators are given by the formula 
\begin{equation}\label{eq3.E}
\bl \al_i,\al_j\br=\sum_{k=1}^L (s^k_{ij} (\al_k\otimes 1)- s^k_{ji} (1\otimes \al_k)), \qquad i,j=1,\dots,L, \quad s^k_{ij}\in\kk.
\end{equation}
As shown in \cite[Proposition 10]{pich}, this happens if and only if the quantities $s^k_{ij}$ serve as the structure constants of an associative multiplication $\mu$ on the vector space $\kk^L$. The KKS bracket from Example \ref{examp_2} is a particular case, when the corresponding associative algebra is the direct sum of $L$ copies of the base field $\kk$.

Recall the definition of the involutions $\phi^+$ and $\phi^-$ from Example \ref{examp_1}.

\begin{proposition}\label{prop2}
The bracket \eqref{eq3.E} is $\phi^-$-adapted if and only if $s^k_{ij}=s^k_{ji}$, meaning that the corresponding multiplication $\mu$ on $\kk^L$ is commutative.
\end{proposition}

\begin{proof}
By virtue of Lemma \ref{lemma3.A} it suffices to check the relation \eqref{eq3.A} for the generators of the free algebra, that is, when $a=\al_i$, $b=\al_j$ for some $i,j\in\{1,\dots,L\}$. From \eqref{eq3.E} we see that the left-hand side of \eqref{eq3.A} equals
$$
\sum_{k=1}^N (s^k_{ij} (\phi^-(\al_k)\otimes 1)- s^k_{ji} (1\otimes \phi^-(\al_k)))=\sum_{k=1}^N (-s^k_{ij} (\al_k\otimes 1)+s^k_{ji} (1\otimes \al_k)),
$$
while the right-hand side equals
$$
\bl \phi^-(\al_i),\phi^-(\al_j)\br^\circ=\bl \al_i,\al_j\br^\circ=\sum_{k=1}^N (- s^k_{ji} (\al_k \otimes 1)+s^k_{ij} (1\otimes \al_k)).
$$
This means that the equality is achieved if and only if $s^k_{ij}=s^k_{ji}$. 
\end{proof}

\begin{remark}
For the involution $\phi^+$ from Example \ref{examp_1}, the same computation leads to the relation $s^k_{ij}=-s^k_{ji}$. It means that the corresponding multiplication $\mu$ on $\kk^L$ should be skew-symmetric. This excludes at once unital algebras, but non-unital associative algebras with skew-symmetric multiplication rule exist, although they seem to be exotic. For instance, one can take an arbitrary two-step nilpotent Lie algebra. Then its bracket also defines an associative mutiplication rule for the trivial reason that any triple product equals $0$.
\end{remark}

\subsection{Second example: quadratic $\phi$-adapted brackets}\label{sect4.3}

Here we use some results from the paper \cite{ORS} by Odesskii, Rubtsov, and Sokolov (see also Rubtsov and Such\'anek   \cite{RubSuch}, section 12.3). Denote by $e_1,\dots,e_L$ the standard basis in $\kk^L$ and let $\mathcal R: \kk^L\otimes \kk^L\to \kk^L\otimes \kk^L$ be given by a tensor $r^{kl}_{ij}$:
$$
\mathcal R(e_i\otimes e_j)=\sum_{k,l=1}^L r^{kl}_{ij}\, e_k\otimes e_l.
$$
It is easily seen  that the formula 
\begin{equation}\label{eq3.D}
\bl \al_i,\al_j\br=\sum_{k,l=1}^L r^{kl}_{ij} \,\al_k\otimes \al_l
\end{equation}
determines a double Poisson bracket on $\kk\langle \al_1,\dots,\al_L\rangle$ if and only if the following two conditions hold:

\begin{itemize}
\item  the tensor $r_{ij}^{kl}$ is skew-symmetric with respect to simultaneous swapping of the two upper indices and the two lower indices,
that is, $r^{kl}_{ij}=-r^{lk}_{ji}$; this can also be written as
$$
\mathcal R^{12}=-\mathcal R^{21};
$$
\item  the operator $\mathcal R$ satisfies the relation
$$
\mathcal R^{12}\mathcal R^{23}+\mathcal R^{23}\mathcal R^{31} +\mathcal R^{31}\mathcal R^{12}=0.
$$
\end{itemize}
Here the symbol $\mathcal R^{ab}$ means that $\mathcal R$ acts on the $a$th and $b$th factors in a multiple tensor product. 

Due to skew-symmetry, the latter relation can be rewritten as 
\begin{equation}\label{eq4.A}
\mathcal R^{12}\mathcal R^{23}-\mathcal R^{23}\mathcal R^{13} -\mathcal R^{13}\mathcal R^{12}=0.
\end{equation}
Next, applying the transposition of the first and third tensor factors and using skew-symmetry again, this can be further rewritten as  \begin{equation}\label{eq4.B}
\mathcal R^{12}\mathcal R^{13}-\mathcal R^{23}\mathcal R^{12} +\mathcal R^{13}\mathcal R^{23}=0
\end{equation}
(see e.g. Schedler \cite{Schedler}, item (i) of Theorem 2.8). Both versions, \eqref{eq4.A} and \eqref{eq4.B} are known in the literature as the \emph{associative Yang-Baxter equation}, or AYBE for short. 

Thus, each skew-symmetric $L\times L$ matrix solution of the AYBE produces a double Poisson bracket on $\kk\langle \al_1,\dots,\al_L\rangle$ --- a \emph{quadratic} bracket, in the terminology of \cite{ORS} (but note that there are quadratic brackets of a more general kind, see \cite{ORS}). 

\begin{lemma}
Let $\phi$ be any of the two involutions $\phi^+$, $\phi^-$ from Example \ref{examp_1}. A quadratic double Poisson bracket \eqref{eq3.D} is $\phi$-adapted if and only if the tensor $r^{kl}_{il}$ is symmetric with respect to the upper indices: $r^{kl}_{ij}=r^{lk}_{ij}$.
\end{lemma}

\begin{proof}
Evident.
\end{proof}

\begin{corollary}
Take as $\phi$ one of the involutions $\phi^+$, $\phi^-$ from Example \ref{examp_1}. Suppose that the tensor $r^{kl}_{ij}$, where $i,j,k,l$ range over $\{1,\dots,L\}$, is symmetric in the upper indices, skew symmetric in the low indices, and the corresponding operator $\mathcal R$ satisfies the AYBE. Then the formula \eqref{eq3.D} determines a $\phi$-adapted quadratic double Poisson bracket on $\kk\langle\al_1,\dots,\al_L\rangle$.  
\end{corollary}

Here is a concrete example of such brackets, extracted from \cite[(2.19)]{ORS}: 
\begin{gather*}
\bl \al_i,\al_i\br=0,\\
\bl \al_i,\al_j\br=\frac1{\la_i-\la_j}( \al_i\otimes \al_j+\al_j\otimes \al_i -\al_i\otimes \al_i-\al_j\otimes \al_j ), \quad i\ne j,
\end{gather*}
where $\la_1,\dots,\la_L$ is an $N$-tuple of pairwise distinct parameters. 

\section{Poisson brackets from the centralizer construction}\label{sect5}

\subsection{Preliminaries} 

Here is a brief description of what the centralizer construction is.  One takes an ascending chain $\{\aa(N)\}$ of reductive Lie algebras and a chain $\{\bb(N)\}$ of their subalgebras, and one considers the centralizer algebras 
$$
U(\aa(N))^{\bb(N)} \subset U(\aa(N))
$$
inside the universal enveloping algebras $U(\aa(N))$. In some concrete cases one can define filtration preserving algebra morphisms 
$$
U(\aa(N-1))^{\bb(N-1)}\leftarrow U(\aa(N))^{\bb(N)},
$$
which makes it possible to define a projective limit algebra

\begin{equation}\label{eq5.A}
\varprojlim\, U(\aa(N))^{\bb(N)}, \quad N\to\infty,
\end{equation}
where the limit is taken in the category of filtered associative algebras. When this procedure works, it gives a sense to the concept of \emph{large-$N$ limit of centralizer algebras}.

In the initial papers, the centralizer construction was studied in the following context: $\{\aa(N)\}$ is one of the $4$ series of the classical reductive Lie algebras and the subalgebra $\bb(N)\subset\aa(N)$ belongs to the same series, but has a smaller rank, differing from the rank of $\aa(N)$ by a fixed positive integer, say $d$. This led to an alternative approach to the Yangian of the Lie algebra $\gl(d,\kk)$ (when $\aa(N)=\gl(N,\kk)$) and the discovery of the so-called twisted Yangians (when the $\aa(N)$'s are the orthogonal or symplectic Lie algebras): all these Yangian algebras sit inside the limit algebras \eqref{eq5.A}. See \cite{Ols1989}, \cite{Ols1992}, and a more detailed exposition in \cite{Ols1991}, \cite{MO}, \cite[chapter 8]{M}.

A different version of the centralizer construction is investigated in the recent paper \cite{Ols}. Here one takes 
\begin{equation}\label{eq5.C}
\begin{aligned}
\aa(N)=\gl(N,\kk)^{\oplus L}=\gl(N,\kk)\oplus\dots\oplus\gl(N,\kk) \quad \text{($L$ times)},\\
\bb(N)=\gl_d(N,\kk):=\operatorname{span}\{E_{kl}\colon d+1\le k,l\le N\}\simeq \gl(N-d,\kk)\subset \gl(N,\kk),
\end{aligned}
\end{equation}
with the understanding that $\bb(N)\to\aa(N)$ is the diagonal embedding and $L\ge2$ is a fixed integer. In this situation, a projective limit algebra  \eqref{eq5.A} can again be constructed, and it contains a ``Yangian-type'' subalgebra denoted by $Y_{d,L}$.  

Next, the centralizer construction is also applicable to symmetric algebras.  That is, a noncommutative filtered limit algebra of the form \eqref{eq5.A} can be replaced by a commutative graded limit algebra of the form
\begin{equation}\label{eq5.B}
\varprojlim\, S(\aa(N))^{\bb(N)}, \quad N\to\infty,
\end{equation}
which possesses a natural Poisson structure. Note that working with the Poisson algebras \eqref{eq5.B} is much easier than with the noncommutative associative algebras \eqref{eq5.A}. 

In the context of \eqref{eq5.C}, the Poisson structure of the limit algebra \eqref{eq5.B} was described in \cite[section 5]{Ols}. There it was shown that the Poisson bracket is related to the KKS double Poisson bracket on the free algebra with $L$ generators. Our aim is to extend the computation in  \cite[Proposition 5.2]{Ols}  to the case when the general linear Lie algebras in \eqref{eq5.C} are replaced by the orthogonal or symplectic Lie algebras. The result is given by Proposition \ref{th4}. As mentioned above (section \ref{sect1.5}), it  explains the origin of our main result. The computation given below is direct and simple enough, and it can be extended to other cases. This is one more reason why we decided to include this material in the present paper.

\subsection{Poisson brackets}

Let us set $\mathfrak{g}_N=\mathfrak{o}(N)$ or $\mathfrak{g}_N=\mathfrak{sp}(N)$ (the orthogonal or symplectic Lie algebra over $\kk$, of rank $r:=\lfloor N/2\rfloor$). It will be convenient for us to work with the following realization of $\mathfrak{g}_N$. Let us rename the standard basis $\{e_1,\dots,e_N\}$ of $\kk^N$ to $\{e_{-r},e_{-r+1},\ldots, e_{r-1},e_{r}\}$, where $e_0$ is dropped for $N$ even. Introduce also a map $\theta:\mathbb{Z}\rightarrow \{\pm1\}$ such that $\theta(i)\equiv1$  in the orthogonal case and $\theta(i)=\operatorname{sgn}(i)$ in the symplectic case. In what follows the indices $i,j,k,l$ are always taken from the set $\{-\lfloor N/2\rfloor, \dots,\lfloor N/2\rfloor\}$ with the understanding that $0$ is excluded for $N$ even.

Using this notation we realize $\mathfrak{g}_N$ as the subalgebra of $\gl(N,\kk)$ spanned by the elements
\begin{equation*}
    F_{ij}:=E_{ij}-\theta(ij)E_{-j,-i}.
\end{equation*}

\begin{remark}\label{rem1}
    Note that $\mathfrak{g}_N=\{x\in\mathfrak{gl}(N,\kk)\mid \tau(x)=-x\}$, where $\tau$ is the matrix involution corresponding to the symmetric (for $\mathfrak{o}(N)$) or skew-symmetric (for $\mathfrak{sp}(N)$) bilinear form $\left<e_i,e_j\right>=\theta(i)\delta_{i,-j}$, so that $\tau$ is given by
\begin{equation*}
    \tau(E_{ij})=\theta(ij)E_{-j,-i}.
\end{equation*}
\end{remark}

The elements $F_{ij}$ possess the following symmetry
\begin{equation}\label{f11}
    F_{ij}=-\theta(ij)F_{-j,-i}
\end{equation}
and obey the commutation relations
\begin{equation}\label{f8}
    \left[F_{ij},F_{kl}\right]=\delta_{kj}F_{il}-\delta_{il}F_{kj}-\theta(ij)\Bigl(\delta_{k,-i}F_{-jl}-\delta_{-jl}F_{k,-i}\Bigr).
\end{equation}

If we set $Q^{ij}_{kl}:=\delta_{kj}F_{il}-\delta_{il}F_{kj}$, then the commutation relations \eqref{f8} can be rewritten as
\begin{equation}\label{eq5.F}
    \left[F_{ij},F_{kl}\right]=Q^{ij}_{kl}-\theta(ij)Q^{-j,-i}_{kl}.
\end{equation}

Pick a positive integer $L$. Below we will work with the symmetric algebra $S\left(\mathfrak{g}_N^{\oplus L}\right)$. The element $F_{ij}$ from the $r$-th component of $\mathfrak{g}_N^{\oplus L}$ will be denoted by $F_{ij\mid r}$.

We regard the set $[L]:=\{1,2,\ldots,L\}$ as an alphabet with $L$ letters and denote by $W_L$ the set of all words in this alphabet. We are dealing with the free algebra in $L$ generators $\kk\langle \al_1,\ldots,\al_L\rangle$. To simplify the notation we identify the generators $\al_i$ with the corresponding indices $i$, which are treated as letters from the alphabet $[L]$. This enables us to identify monomials in the generators with words from $W_L$. Thus, multiplication of monomials in the free algebra is the same as concatenation of words. The empty word corresponds to the unit element of the algebra. 

We consider the double Poisson bracket on $\kk\langle \al_1,\ldots,\al_L\rangle$ from Example \ref{examp_2}. Recall that it is defined  on the generators by 
\begin{equation}\label{eq5.D}
    \bl \al_i, \al_j\br=\de_{ij}(1\otimes \al_i-\al_i\otimes 1), \quad i,j=1,\dots,L.
\end{equation}
More generally, for arbitrary words $w,t\in W_L$ the bracket is given by 
\begin{equation}\label{eq5.E}
    \{\!\!\{w,t\}\!\!\}=\sum\limits_{\substack{\alpha\in \left[L\right],\ w',w'',t',t''\colon\\
    w=w'\alpha w'',\ t=t'\alpha t''}}(t'w''\otimes w'\alpha t''-t'\alpha w''\otimes w' t'').
\end{equation}
Note that some of the words $w', w'', t', t''$ may be empty, which happens if the letter $\al$ is at the very beginning or at the very end of $w$ or $t$. Thus, if both $t'$ and $w''$ are empty (meaning that $\al$ is the first letter of $t$ and also the last letter of $w$), then  $t'w''=1$; a similar convention is applied to $w't''$. The formula \eqref{eq5.E} is obtained from \eqref{eq5.D} by applying the Leibniz rule. 

Finally, given a word $w=w_1\dots w_m\in W_L$,  let $f_{ij}(w)$ denote the following expression
\begin{equation}\label{f9}
    f_{ij}(w):=\sum\limits_{a_1,\ldots,a_{m-1}}F_{ia_1\mid w_1}F_{a_1a_2\mid w_2}\cdot\ldots\cdot F_{a_{m-1}j\mid w_m}\in S(\mathfrak{g}_N^{\oplus L}).
\end{equation}
Here $i,j$ are taken from the set $\{-\lfloor N/2\rfloor, \dots,\lfloor N/2\rfloor\}$ (where $0$ is excluded if $N$ is even), and the indices $a_1,\ldots,a_{m-1}$ range over the same set. We also set $f_{ij}(\varnothing):=\de_{ij}$ for the empty word $w=\varnothing$. 

Note that if we additionally impose the restriction $-d\le i, j \le d$, then the corresponding elements \eqref{f9} commute with the subalgebra $\bb(N)$ which is isomorphic to $\mathfrak o(2(N-d))$ or $\mathfrak{sp}(2(N-d))$. 

In the next proposition we extend the symbol $f_{ij}(\cdots)$ to linear combinations of words by linearity. 

\begin{proposition}\label{th4}
Let $\{-,-\}$ be the canonical Poisson bracket in the symmetric algebra $S(\mathfrak{g}_N^{\oplus L})$. The values of this bracket on the elements of the form \eqref{f9} are given by
\begin{multline}\label{f10}
    \{f_{ij}(w),f_{kl}(t)\}=f_{kj}\left(\{\!\!\{w,t\}\!\!\}'\right)f_{il}\left(\{\!\!\{w,t\}\!\!\}''\right)\\
    +\theta(ij)\cdot f_{k,-i}\left(\{\!\!\{\phi^-(w),t\}\!\!\}'\right)f_{-j,l}\left(\{\!\!\{\phi^-(w),t\}\!\!\}''\right),
\end{multline}
where $w,t\in\kk\langle \al_1,\ldots,\al_L\rangle$ and $\phi^-$ is the antiautomorphism of the free algebra specified in Example \ref{examp_1}.
\end{proposition}

\begin{remark}
A link between Proposition \ref{th4} and our main result, Theorem \ref{thm3.A}, is established as follows.  In that theorem, we take $A:=\kk\langle \al_1,\ldots,\al_L\rangle$ and $\phi:=\phi^-$, and we consider the double Poisson bracket \eqref{eq5.E} on $A$. As the matrix involution $\tau$, we take the one from Remark \ref{rem1}.
Then the Poisson bracket $\{-,-\}_{\phi,\tau}$ on $\OO(A,d)^{\phi,\tau}$ reads as
\begin{equation*}
    \{w_{ij},t_{kl}\}_{\phi,\tau}=\{\!\!\{w,t\}\!\!\}'_{kj}\{\!\!\{w,t\}\!\!\}''_{il}+\theta(ij)\cdot\{\!\!\{\phi^-(w),t\}\!\!\}'_{k,-i}\{\!\!\{\phi^-(w),t\}\!\!\}''_{-j,l}
\end{equation*}
for any $w,t\in\kk\langle \al_1,\ldots,\al_L\rangle$ and any $i,j,k,l$. We see that this expression coincides with \eqref{f10}, up to the identification $f_{ij}(w)\leftrightarrow w_{ij}$. 
\end{remark}

\begin{proof}[Sketch of proof of Proposition \ref{th4}]
Let $w=w_1\ldots w_m$ and $t=t_1\ldots t_n$ be the decompositions into letters. Set $[m]:=\{1,\dots,m\}$, $[n]:=\{1,\dots,n\}$. 
By the Leibniz rule we can write
\begin{equation}\label{eq5.H}
    \{f_{ij}(w),f_{kl}(t)\}=\sum\limits_{\substack{a_0,\ldots,a_m,\\ b_0,\ldots, b_n\\
    a_0=i,\ a_m=j,\\ b_0=k,\ b_n=l}}\sum_{p\in\left[m\right],q\in\left[n\right]}\Big\{F_{a_{p-1},a_p\mid w_p},F_{b_{q-1},b_q\mid t_q}\Big\}\prod\limits_{\substack{r=1\\r\neq p}}^{m}F_{a_{r-1},a_r\mid w_r}\prod\limits_{\substack{s=1\\s\neq q}}^{n}F_{b_{s-1},b_s\mid t_s}.
\end{equation}
The bracket $\{-,-\}$ on the right-hand side is given by the same formula as the commutator in \eqref{eq5.F}:
\begin{equation}\label{eq5.G}
    \Big\{F_{a_{p-1},a_p\mid w_p},F_{b_{q-1},b_q\mid t_q}\Big\}=\delta_{w_p,t_q}\Bigg(Q_{b_{q-1},b_q}^{a_{p-1},a_p}\Big\vert_{w_p}-\theta(a_{p-1}a_p)Q_{b_{q-1},b_q}^{-a_p,-a_{p-1}}\Big\vert_{w_p}\Bigg),
\end{equation}
where the vertical bar on the right-hand side has the same meaning as in the notation $F_{ij\mid r}$. 
 
It is straightforward to check that the contribution to the right-hand side of \eqref{eq5.H} coming from the first summand in \eqref{eq5.F} results in the first summand in \eqref{f10}. For  more details we refer the reader to the discussion at the end of the proof of \cite[Proposition 5.2]{Ols}. 

Now we focus on the contribution from the second summand. We are going to show that it produces a similar result after a change: 
$$
i\mapsto -j, \quad j\mapsto -i, \quad w\mapsto \phi^-(w)=(-1)^m w_m\ldots w_1.
$$

Indeed, we have to handle the expression
$$
\sum\limits_{\substack{a_0,\ldots,a_m,\\b_0,\ldots, b_n\\ a_0=i,\ a_m=j,\\ b_0=k,\ b_n=l}}\sum_{p\in\left[m\right],q\in\left[n\right]}\delta_{w_p,t_q}(-1)\theta(a_{p-1}a_p)Q_{b_{q-1},b_q}^{-a_p,-a_{p-1}}\Big\vert_{w_p}\prod\limits_{\substack{r=1\\r\neq p}}^{m}F_{a_{r-1},a_r\mid w_r}\prod\limits_{\substack{s=1\\s\neq q}}^{n}F_{b_{s-1},b_s\mid t_s}.
$$
Applying the symmetry relation \eqref{f11} we transform it to
\begin{multline*}
\sum\limits_{\substack{a_0,\ldots,a_m,\\b_0,\ldots, b_n\\ a_0=i,\ a_m=j,\\ b_0=k,\ b_n=l}}\sum_{p\in\left[m\right],q\in\left[n\right]}\delta_{w_p,t_q}(-1)\theta(a_{p-1}a_p)Q_{b_{q-1},b_q}^{-a_p,-a_{p-1}}\Big\vert_{w_p}
    \\
\times \prod\limits_{\substack{r=1\\r\neq p}}^{m}\Bigg((-1)\theta(a_{r-1}a_r)F_{-a_r,-a_{r-1}\mid w_r}\Bigg)\prod\limits_{\substack{s=1\\s\neq q}}^{n}F_{b_{s-1},b_s\mid t_s}.
\end{multline*}
This can be further written as 
\begin{multline*}
 \sum\limits_{\substack{a_0,\ldots,a_m,\\b_0,\ldots, b_n\\ a_0=i,\ a_m=j,\\ b_0=k,\ b_n=l}}\sum_{p\in\left[m\right],q\in\left[n\right]}\delta_{w_p,t_q}(-1)^m\theta(a_0a_1\cdot a_1a_2\cdot\ldots \cdot a_{m-1}a_m)
    \\
 \times F_{-a_m,-a_{m-1}\mid w_m}F_{-a_{m-1},-a_{m-2}\mid w_{m-1}}\ldots Q_{b_{q-1},b_q}^{-a_p,-a_{p-1}}\Big\vert_{w_p} F_{-a_{p-1},-a_{p-2}\mid w_{p-1}}\ldots F_{-a_1,-a_0\mid w_1}
    \\
    \times \prod\limits_{\substack{s=1\\s\neq q}}^{n}F_{b_{s-1},b_s\mid t_s}.
\end{multline*}
Finally, we note that 
$$
(-1)^m\theta(a_0a_1\cdot a_1a_2\cdot\ldots \cdot a_{m-1}a_m)=(-1)^m\theta(ij)
$$ 
and relabel the indices $a_1\mapsto -a_1,\ldots, a_{m-1}\mapsto -a_{m-1}$. After this minor change the remaining part of the last expression takes the desired form, leading to the second summand in \eqref{f10}. 
\end{proof}

\subsection*{Acknowledgements}
We are grateful to Maxime Fairon for productive discussions, many useful comments, and suggestions. The second named author (N.S.) wishes to express his sincere gratitude to Maria Gorelik and Dmitry Gourevitch for their  hospitality at the Weizmann Institute of Science. N.S. is immensely appreciative of Michael Pevzner for his invaluable support, which made it possible for N.S. to work at the University of Reims Champagne-Ardenne under the PAUSE program. N.S. is grateful to everyone at the Mathematical Laboratory of Reims, especially to Valentin Ovsienko and Sophie Morier-Genoud, for their kind support and welcoming atmosphere.

\subsection*{Funding}
The present work was supported by the Russian Science Foundation under project 23-11-00150.


\begin{thebibliography}{AA}

\bibitem{AKKN}
Alekseev, A., Kawazumi, N., Kuno, Y.,  Naef, F.: The Goldman-Turaev Lie bialgebra in genus zero and the Kashiwara–Vergne problem. Adv. Math. 326, 1-53 (2018). 

\bibitem{LeBruyn}
Le Bruyn, L.: Noncommutative geometry and Cayley-smooth orders. Chapman and Hall/CRC (2008).

\bibitem{Etingof}
Crawley-Boevey, W., Etingof, P., Ginzburg, V.: Noncommutative geometry and quiver algebras, Advances in Mathematics 209 (1), 274-336  (2007).

\bibitem{FaironMcCulloch}
Fairon, M., McCulloch, C.: Around Van den Bergh’s double brackets for different bimodule structures, Communications in Algebra 51(4), 1673-1706 (2023).

\bibitem{FaironMorphism}
Fairon, M.: Morphisms of double (quasi-) Poisson algebras and action-angle duality of integrable systems, Annales Henri Lebesgue 5, 179-262 (2022).

\bibitem{M}
Molev, A.: Yangians and classical Lie algebras. Mathematical Surveys and Monographs 143. Amer. Math. Soc. (2007).

\bibitem{MO}
Molev, A., Olshanski, G.: Centralizer construction for twisted Yangians. Selecta Math. 6, 269-317  (2000). 


\bibitem{ORS}
Odesskii, A. V.,  Rubtsov, V. N., Sokolov, V.V.:  Double Poisson brackets on free associative algebras. In:  Noncommutative birational geometry, representations and combinatorics. Contemp. Math., vol. 592, Amer. Math. Soc., Providence, RI,  pp. 225--239 (2013).

\bibitem{Ols1989}
Olshanskii, G. I.: Yangians and universal enveloping algebras. J. Soviet  Math. 47, 2466-2473 (1989).

\bibitem{Ols1991}
Olshanskii, G. I.: Representations of infinite-dimensional classical groups, limits of enveloping algebras, and Yangians. In: Topica in Representation Theory (A. A. Kirillov, ed.). Advances in Soviet Math. 2. Amer. Math. Soc., Providence, R. I., pp. 1-66 (1991). 


\bibitem{Ols1992}
Olshanskii, G. I.: Twisted Yangians and infinite-dimensional Lie algebras. In: Quantum groups. Proceedings of Workshops held in the Euler Mathematical Institute, Leningrad, Fall 1990 (P. P. Kulish, ed.). Lecture Notes in Math. 1510. Springer, pp. 104-119 (1992).

\bibitem{Ols}
Olshanski, G.: The centralizer construction and Yangian-type algebras. Journal of Geometry and Physics 196, 105063 (2024). 

\bibitem{Schedler} 
Schedler, T.: Poisson algebras and Yang-Baxter equations. In: Advances in Quantum Computation. Contemp. Math. 482.  Amer. Math. Soc., Providence, R. I., pp. 91-106 (2009).

\bibitem{vdB}
Van den Bergh, M.: Double Poisson algebras, Trans. Amer. Math. Soc. 360, 5711-5799  (2008).

\bibitem{pich}
Pichereau, A.,  Van de Weyer, G.:  Double Poisson cohomology of path algebras of quivers. Journal of Algebra 319 (5), 2166-2208 (2008).

\bibitem{RubSuch}
Rubtsov, V.,  Suchánek, R.:  Lectures on Poisson algebras. In: Groups, Invariants, Integrals, and Mathematical Physics. The Wisła 20-21 Winter School and Workshop. Birkh\"{a}user, pp. 41-116 (2023).

\end{thebibliography}
\end{document}